\documentclass[11pt]{amsart}
\usepackage{amssymb}
\usepackage{amsmath}
\usepackage[hidelinks]{hyperref}
\usepackage{etoolbox}

\newtheorem{theorem}{Theorem}[section]
\newtheorem{lemma}[theorem]{Lemma}

\newtheorem{proposition}[theorem]{Proposition}
\newtheorem{fact}[theorem]{Fact}

\newtheorem{corollary}[theorem]{Corollary}
\newtheorem{conjecture}[theorem]{Conjecture}

\theoremstyle{remark}
\newtheorem{remark}[theorem]{Remark}

\numberwithin{equation}{section}

\newcommand{\mhp}{\vskip -.3cm}
\newcommand{\qp}{\vskip .12cm}
\newcommand{\hp}{\vskip .2cm}
\newcommand{\p}{\vskip .4cm}

\newcommand{\Z}{\mathbb{Z}}
\newcommand{\Q}{\mathbb{Q}}
\newcommand{\R}{\mathbb{R}}
\newcommand{\C}{\mathbb{C}}

\newcommand{\Ql}{\bar{\mathbb{Q}}_{\ell}}
\newcommand{\CA}{\mathcal{A}}
\newcommand{\CB}{\mathcal{B}}
\newcommand{\CL}{\mathcal{L}}
\newcommand{\CT}{\mathcal{T}}
\newcommand{\CX}{\mathcal{X}}

\newcommand{\CO}{\mathcal{O}}

\newcommand{\bG}{\mathbf{G}}
\newcommand{\bT}{\mathbf{T}}
\newcommand{\bS}{\mathbf{S}}
\newcommand{\bH}{\mathbf{H}}
\newcommand{\Ga}{\mathbb{G}_a}
\newcommand{\til}{\tilde}
\newcommand{\bsl}{\backslash}
\newcommand{\ra}{\rightarrow}
\newcommand{\sra}{\twoheadrightarrow}
\newcommand{\xra}{\xrightarrow}

\newcommand{\Lie}{\operatorname{Lie}}
\newcommand{\rank}{\operatorname{rank}}
\newcommand{\depth}{\operatorname{depth}}
\newcommand{\Ad}{\operatorname{Ad}}
\newcommand{\ad}{\operatorname{ad}}
\newcommand{\Spec}{\operatorname{Spec}}
\newcommand{\Frob}{\mathrm{Frob}}
\newcommand{\Irr}{\operatorname{Irr}}
\newcommand{\Tr}{\operatorname{Tr}}
\newcommand{\Gal}{\operatorname{Gal}}
\newcommand{\Id}{\operatorname{Id}}
\newcommand{\Lg}{\mathfrak{g}}
\newcommand{\Lu}{\mathfrak{u}}
\newcommand{\Lt}{\mathfrak{t}}

\newcommand{\matr}[1]{\left[\begin{matrix}#1\end{matrix}\right]}
\newcommand{\dsp}{\displaystyle}

\begin{document}

\title{Components of affine Springer fibers}
\author{Cheng-Chiang Tsai}
\thanks{This work is supported by National Science Foundation [DMS-1128155] and [DMS-1601282].}
\email{chchtsai@mit.edu}
\address{77 Massachusetts Avenue, Massachusetts Institute of Technology, Cambridge, MA 02139, USA}
\begin{abstract}
Let $\bG$ be a connected split reductive group over a field of characteristic zero or sufficiently large characteristic, $\gamma_0\in(\Lie\bG)((t))$ be any topologically nilpotent regular semisimple element, and $\gamma=t\gamma_0$. Using methods from $p$-adic orbital integrals, we show that the number of components of the Iwahori affine Springer fiber over $\gamma$ modulo $Z_{\bG((t))}(\gamma)$ is equal to the order of the Weyl group.
\end{abstract}
\makeatletter
\patchcmd{\@maketitle}
  {\ifx\@empty\@dedicatory}
  {\ifx\@empty\@date \else {\vskip3ex \centering\footnotesize\@date\par\vskip1ex}\fi
   \ifx\@empty\@dedicatory}
  {}{}
\patchcmd{\@adminfootnotes}
  {\ifx\@empty\@date\else \@footnotetext{\@setdate}\fi}
  {}{}{}
\makeatother
\maketitle 

\tableofcontents

\section{Introduction}

In classical Springer theory, a Springer representation is realized on the top (co)homology of a Springer fiber, which is the same as the vector space generated by a basis indexed by its irreducible components. Springer theory also computes this vector space in terms of the (already computed, see e.g. \cite{Ca93}) Springer representations. As an affine generalization, the affine Weyl group also acts on the homology of an Iwahori affine Springer fiber \cite{Lu96}. One then wonders what can be said about the components of Iwahori affine Springer fibers.\p

In this article, we fix $\bG$ a connected split reductive group over a field $k$ of characteristic zero or sufficiently large characteristic (see Appendix \ref{secchar}). In fact, in Appendix \ref{appred} we explain how to reduce our result over a general field to the case over a finite field, and from now on we assume $k$ is a finite field. We also use throughout the article the notations $F=k((t))$, $\CO=k[[t]]$, $G=\bG(F)$ and $\Lg=(\Lie\bG)(F)$, as well as take $\bar{k}$ an algebraic closure of $k$ and $F^{ur}$ a maximal unramified extension of $F$ with residue field $\bar{k}$. For $\gamma\in\Lg$, the Iwahori affine Springer fiber $\CX_{\gamma}$ over $\gamma$ (\cite{KL88}, see (\ref{ASF})) is an ind-variety with $\CX_{\gamma}(k)=\{g\in G/I\,|\,\Ad(g^{-1})\gamma\in\Lie I\}$ where $I\subset G$ is an Iwahori subgroups; we refer the readers to Section \ref{secBT} for the definitions where we put $I=G_{x,\ge 0}$ and $\Lie I=\Lg_{x,\ge0}$. We have $\CX_{\gamma}$ is finite-dimensional iff $\gamma$ is regular semisimple (\cite[\S3]{KL88}, and see the proof of Lemma \ref{equi} for the non-simply-connected case). Suppose this is the case. A dimension formula for $\CX_{\gamma}$ was conjectured in \cite[\S0]{KL88} and proved by Bezrukavnikov \cite{Be96} (see (\ref{Bezr})). As $\gamma$ approaches $0$, $\dim\CX_{\gamma}$ grows unboundedly. It can be natural to ask whether the number of components of $\CX_{\gamma}$, say modulo the natural centralizer action to make it finite, also grows unboundedly. It turns out that the number remains constant when $\gamma$ approaches $0$.\p

More precisely, let $\gamma_0$ be a topologically nilpotent regular semisimple element in $\Lg$. Here $\gamma_0$ is called topologically nilpotent (see Lemma \ref{TN}) if it is conjugate to an element $\delta\in(\Lie\bG)(k[[t]])$ whose reduction in $(\Lie\bG)(k)$ is nilpotent. Let $\gamma=t\gamma_0$. Denote by $W$ the Weyl group of $\bG$. The main result of this article is

\begin{theorem}\label{main} Let $\gamma$ be as above. Consider the set of irreducible components of $\CX_{\gamma}\times_{\Spec k}\Spec\bar{k}$ under the action of the centralizer of $\gamma$. There are $|W|$ distinct orbits, and they are all stabilized by $\Gal(\bar{k}/k)$.
\end{theorem}\hp

While this is a geometric result, our proof relies on the method of $p$-adic orbital integrals. Based on that a variety over $k\cong\mathbb{F}_q$ with $C$ irreducible components of dimension $d$ over $k$ has ``roughly'' $Cq^d$ rational points, and that the point-count on the affine Spriner fiber over $\gamma$ is related to a (regular semisimple) orbital integral over the orbit $\Ad(G)\gamma$, we re-interprete the dimension formula and the question for the number of components as an estimate for orbital integrals. 
In fact, the formulation of such estimate turns out to be pretty natural and allows us to realize the dimension formula for affine Springer fibers proved by Bezrukavnikov as an affine generalization of that of finite Springer fibers.\p

In our formulation, Bezrukavnikov's dimension formula also has an analogue for nilpotent elements (Proposition \ref{comb}), which we prove by computing nilpotent orbital integrals using Ranga Rao's method \cite{Ra72}. Our results for nilpotent orbital integrals achieve sharper estimates than the regular semisimple case. DeBacker's homogeneity result \cite{De02a} allows us to compare our regular semisimple orbital integral with nilpotent ones to obtain a sharper estimate also for the former. The estimate is then strong enough to imply Theorem \ref{main}.\p

We explain the structure of this article. In Section \ref{secBT} we review DeBacker's results that we will use, preceded by a quick tour in Bruhat-Tits theory. Note that we use a different notation of Moy-Prasad filtration, etc, than the usual notation by Moy and Prasad. In Section \ref{secnorm}, we introduce various normalizations for orbital integrals, and at the end discuss what we mean by ``estimates'' of orbital integrals. In Section \ref{secgeom}, we introduce affine Springer fibers and several generalization as well as their basic properties. We also explain how estimates on their point-counting are related component-counting. In Section \ref{secdim}, we review the dimension formula for affine Springer fibers \cite{KL88}, \cite{Be96}, re-interpret it as estimates for regular semisimple orbital integrals, and state the analogous result for nilpotent orbital integrals. In Section \ref{secSha}, we show how the homogeneity result of DeBacker and certain estimates for orbital integrals, some to be proven in Section \ref{secRR}, can be used to obtain Theorem \ref{main}. In Section \ref{secRR}, Ranga Rao's method \cite{Ra72} is discussed in detail and applied to obtain the needed results for nilpotent orbital integrals. This finishes the proof of Theorem \ref{main}. Section \ref{secSt} provides additional descriptions of the components and reveals that the components behave similarly to that of the Steinberg variety. We also present there two further conjectures motivated by the descriptions. Finally, in Appendix \ref{secchar} the assumption on $\mathrm{char}(k)$ is explained, and in Appendix \ref{appred} we explain how to reduce the result over a general field $k$ to a finite field.

\subsection*{Acknowledgment} The author is grateful to Zhiwei Yun for many inspiring discussions. He also wishes to thank Stanford University for their hospitality during his visit. In addition, he is grateful to the referees who have greatly helped to clarify numerous points.\p

\section{A review on Bruhat-Tits theory and some applications of DeBacker}\label{secBT}

In this section, we state and review various notions and results in the direction of Bruhat-Tits theory \cite{BT72,BT84}. Here we will introduce a setup that works only for the equal characteristic case and also restrict to split groups, and refer readers interested in the general case to \cite{Tits}. Our use of Bruhat-Tits theory is mostly about DeBacker's two results \cite{De02a,De02b}; we will state in this section complete statements of what we need. Readers already familiar with the theory may skip the whole section once they note that the notation $G_{y,\ge r}$ (resp. $\Lg_{y,\ge r}$) will be used to denoted what Moy and Prasad wrote as $G_{y,r}$ (resp. $\Lg_{y,r}$) and instead $\Lg_{y,r}:=\Lg_{y,\ge r}/\Lg_{y,>r}$ will be used to denote what is sometimes written as $\Lg_{y,r:r+}$, etc. Likewise the reductive quotient will be denoted $\bG_{y,0}$ (as an algebraic group over $k$) and $G_{y,0}:=\bG_{y,0}(k)$.\p

To begin with, Bruhat-Tits theory asserts a connected contractible polyhedral complex\footnote{That is, a (usually infinite) collection of polyhedrons, where two $n$-dimensional polyhedrons are glued along their $(n-1)$-dimensional faces.}, the (extended) Bruhat-Tits building $\CB_{\bG}$, which has a list of wonderful properties. It is formed by gluing many ``apartments.'' These apartments are indexed by the set of maximal ($F$-)split tori of $\bG\times_{\Spec k}\Spec F$, where recall all such tori are conjugate by $G=\bG(F)$. Fix $\bS$ a maximal $k$-split torus of $\bG$. The base change $\bS\times\Spec F$ of $\bS$ to $F$ is a maximal $F$-split torus of $\bG\times\Spec F$, which corresponds to an apartment $\CA_{\bS}$. (Note also that the base change of $\bS$ to $\CO=k[[t]]$ gives a fixed integral form of $\bS\times\Spec F$.) In this article, the use of the Bruhat-Tits building can actually be restricted to $\CA_{\bS}$, which we now describe.\p

As a topological space we have $\CA_{\bS}:=X_*(\bS)\otimes\mathbb{R}$. Let $\bar{\Phi}\subset X^*(\bS)$ be the set of roots (of $\bG$ with respect to $\bS$). Every $\bar{\alpha}\in\bar{\Phi}$ gives an $\mathbb{R}$-valued function on $\CA_{\bS}$. An affine root $\alpha$ of $\bG$ with respect to $\bS$ is an object of the form $\alpha=\bar{\alpha}+n$ with $n\in\Z$. We identify an affine root as an $\mathbb{R}$-valued function $\CA_{\bS}$. Let $\Phi$ be the set of affine roots. A {\bf hyperplane} on $\CA_{\bS}$ is the zero locus of some $\alpha\in\Phi$. We cut $\CA_{\bS}$ into polyhedrons using these hyperplanes. We talk about interior of a polyhedron as the obvious combinatorial interior, so that $\CA_{\bS}$ is the disjoint union of the interior of each polyhedron.\p

Let us discuss the use of an apartment before describing the building itself: Each root $\bar{\alpha}$ gives an (not unique) isomorphism $\iota_{\bar{\alpha}}:\Ga\xra{\sim}\Lu_{\bar{\alpha}}\subset\Lie\bG$ from $\Ga$ to the root subspace $\Lu_{\bar{\alpha}}$ over $k$. For an affine root $\alpha=\bar{\alpha}+n\in\Phi$ we have the affine root subspace $\Lg_{\alpha}:=\iota_{\bar{\alpha}}(t^nk)\subset\Lg$. For any point $y\in\CA_{\bS}$, consider the (Moy-Prasad) grading $\dsp\Lg=\bigoplus_{r\in\R}\Lg_{y,r}$ given by
\begin{equation}\label{MPgrading}
\Lg_{y,r}:=\left\{\begin{array}{lll}\dsp\bigoplus_{\alpha\in\Phi,\alpha(y)=r}\Lg_{\alpha},&r\not\in\Z.\\\\
\dsp t^r(\Lie\bS)(k)\oplus\bigoplus_{\alpha\in\Phi,\alpha(y)=r}\Lg_{\alpha},&r\in\Z.
\end{array}\right.\end{equation}
We also write $\Lg_{y,\ge r}:=\bigoplus_{r'\ge r}\Lg_{y,r'}$ and $\Lg_{y,>r}:=\bigoplus_{r'>r}\Lg_{y,r'}$; they are $\CO$-submodules of $\Lg$ and are the so-called Moy-Prasad filtration \cite{MP94}. We have $[\Lg_{y,r_1},\Lg_{y,r_2}]\subset\Lg_{y,r_1+r_2}$, and in particular $\Lg_{y,0}$ is a Lie algebra over $k$. One may check that $(\Lie\bS)(k)\subset\Lg_{y,0}$ is a Cartan subalgebra, and the corresponding roots are those $\{\bar{\alpha}\in\bar{\Phi}\;|\;(\bar{\alpha}+n)(y)=0\text{ for some }\bar{\alpha}+n\in\Phi\}$; such $\bar{\alpha}+n$ are those affine roots whose affine root subspaces appear in $\Lg_{y,0}$. The previous set forms a root subsystem of that of $\bG$. Hence we may embed $\Lg_{y,0}$ into $(\Lie\bG)(k)$ as a reductive subalgebra of equal rank.\p

The group version of Moy-Prasad filtration is slightly more complicated. Each root $\bar{\alpha}$ gives an isomorphism $j_{\bar{\alpha}}:\Ga\xra{\sim}\bG_{\bar{\alpha}}\subset\bG$ over $k$, for an affine root $\alpha=\bar{\alpha}+n\in\Phi$ we have the affine root subgroup $G_{\alpha}:=j_{\bar{\alpha}}(t^nk)\subset\bG(F)=G$. Consider also $G_{\alpha+\Z_{\ge 0}}:=j_{\bar{\alpha}}(t^n\CO)$; this is the closure of the subgroup generated by $G_{\alpha+n}$, $n=0,1,2,...$. For any point $y\in\CA_{\bS}$, Bruhat and Tits associate to it a {\bf parahoric subgroup} which we denote by $G_{y,\ge 0}$, defined to be the subgroup of $G$ generated by $\bS(\CO)$ and $\{G_{\alpha+\Z_{\ge 0}}\;|\;\alpha\in\Phi,\;\alpha(y)\ge 0\}$.\p

For any $r\ge 0$, let $G_{y,\ge r}$ be the subgroup generated by $\ker(\bS(\CO)\ra\bS(\CO/t^{\lceil r\rceil}))$ and $\{G_{\alpha+\Z_{\ge 0}}\;|\;\alpha\in\Phi,\;\alpha(y)\ge r\}$. One checks that this defines, as in the Lie algebra case, a decreasing filtration that only jumps at certain discrete $r$. We may thus also put $G_{y,>r}:=G_{y,\ge r+\epsilon}$ for any sufficiently small $0<\epsilon\ll 1$. The quotient $G_{y,\ge0}/G_{y,>0}$, similar to $\Lg_{y,0}=\Lg_{y,\ge 0}/\Lg_{y,>0}$, is the $k$-point of a reductive group $\bG_{y,0}$ (over $k$) which may be realized as the reductive subgroup of $\bG$ of equal rank with the root subsystem $\{\bar{\alpha}\in\bar{\Phi}\;|\;(\bar{\alpha}+n)(y)=0\text{ for some }\bar{\alpha}+n\in\Phi\}$, so that $\Lg_{y,0}=(\Lie\bG_{y,0})(k)$. The group $\bG_{y,0}$ is called the reductive quotient at $y$. For any $r>0$, by matching root subalgebras and root subgroups, as well as the $(\Lie\bS)$-part with the $\bS$-part, we have a natural isomorphism \cite[3.8]{MP94}
\begin{equation}\label{MP}
\Lg_{y,r}\xra{\sim}G_{y,r}:=G_{y,\ge r}/G_{y,>r}.
\end{equation}
One importance of the reductive quotient is that any $\Lg_{y,\ge r}$ (resp. $G_{y,\ge r}$) is normalized by $G_{y,\ge 0}$, and the induced conjugation action of $G_{y,\ge 0}$ on $\Lg_{y,r}$ (resp. $G_{y,r}$) factors through $G_{y,0}$. In particular, the induced action on $\Lg_{y,0}$ (resp. $G_{y,0}$) is the conjugation action of $G_{y,0}=(\bG_{y,0})(k)$ on its Lie algebra (resp. on itself).\p

It follows immediate from definition that if $y$ and $y'$ are in the interior of the same polyhedron, then $G_{y,\ge 0}=G_{y',\ge 0}$, $G_{y,>0}=G_{y',>0}$, $\Lg_{y,\ge 0}=\Lg_{y',\ge 0}$, $\Lg_{y,>0}=\Lg_{y',>0}$ and there is a canonical identification between $\bG_{y,0}$ and $\bG_{y',0}$. A polyhedron of maximal dimension (that is, $\rank\bG$) is called an alcove. We will fix $x$ some point in some alcove on $\CA_{\bS}$. Any such associated parahoric subgroup $G_{x,\ge 0}$ is called an Iwahori subgroup, which we sometimes denote by $I$. We also have $\bG_{x,0}=\bS$. Another specific parahoric subgroup is the one associated to the origin $o\in X_*(\bS)\otimes\mathbb{R}\cong\CA_{\bS}$. One easily sees that $G_{o,\ge 0}=\bG(\CO)$ and $o$ is called a hyperspecial vertex and $G_{o,\ge 0}$ a hyperspecial parahoric subgroup. One also has $\bG_{o,0}=\bG$. Nevertheless, if one chooses a set of positive roots among $\bar{\Phi}$ and let $\bar{\rho}\in X_*(\bS)$ be any cocharacter in the interior of the positive chamber, then for $0<\epsilon\ll 1$ we have $-\epsilon\bar{\rho}$ lies in an alcove. One may suppose $x=-\epsilon\bar{\rho}$ in which case the Iwahori subgroup $G_{x,\ge 0}$ is the preimage under $\bG_{o,\ge0}=\bG(\CO)\sra\bG(k)$ of the Borel subgroup in $\bG(k)$ with respect to the positive roots.\p

More generally, when we deal with any $y\in\CA_{\bS}$ we may usually (thanks to Fact \ref{FactBT}(iv) below) arrange so that $y$ lies in the closure of the alcove containing $x$, just like $o$ does in the previous paragraph. When this is the case, one checks that $G_{x,\ge 0}$ (resp. $\Lg_{x,\ge 0}$) is the preimage of a Borel subgroup (resp. subalgebra) of $G_{y,0}$ (resp. $\Lg_{y,0}$) under $G_{y,\ge0}\sra G_{y,0}$ (resp. $\Lg_{y,\ge 0}\sra\Lg_{y,0}$). In this manner $\bG_{x,0}$ is the abelianization of that Borel subgroup of $\bG_{y,0}$ (and thus a maximal torus).\p

For any $\gamma\in\Lg$, there is the notion of depth given by
\begin{equation}\label{depth}
\depth(\gamma):=\min\{r\in\R\;|\;\Ad(G)\gamma\cap\Lg_{y,\ge r}\not=\emptyset\text{ for some }y\in\CA_{\bS}\}.
\end{equation}
The minimum always exists as a rational number unless $\gamma$ is nilpotent in which case $\depth(\gamma)=+\infty$. We list some basic results about this notion:
\begin{lemma}\label{TN} Under our assumption on $\mathrm{char}(k)$, for any $\gamma\in\Lg$ we have\hp
(i)$\;\;$ The depth of $\gamma$ is equal to that of its semisimple part (in its Jordan decomposition).\qp
(ii)$\;$ $\depth(t\gamma)=\depth(\gamma)+1$.\qp
(iii) The following are equivalent:\qp
\renewcommand{\theenumi}{\alph{enumi}}
\begin{enumerate}
\item The depth $\depth(\gamma)>0$, i.e. for some $g\in G$  we have $\Ad(g)\gamma\in\Lg_{y,>0}$ for some $y\in\CA_{\bS}$.
\item For some $g\in G$ we have $\Ad(g)\gamma\in\Lg_{y,\ge 0}$ for some $y\in\CA_{\bS}$, and the image of $\Ad(g)\gamma$ in $\Lg_{y,0}$ is nilpotent for any such $g$ and $y$.
\item For some $g\in G$ we have $\Ad(g)\gamma\in\Lg_{o,\ge0}=(\Lie\bG)(\CO)$, and the image of $\Ad(g)\gamma$ in $\Lg_{o,0}=(\Lie\bG)(k)$ is nilpotent.
\item The powers $\ad(\gamma)^n\in\operatorname{End}_F(\Lg)$ converge to zero as $n\ra+\infty$.
\end{enumerate}
\noindent In particular, any $\gamma\in\Lg$ of depth $>0$ is called {\bf topologically nilpotent}.
\end{lemma}\hp

Now we describe the two results of DeBacker, both of which are generalizations of results of Waldspurger. The first result \cite{De02b} parametrizes $G$-orbits of nilpotent elements in $\Lg$ using nilpotent orbits in $\Lg_{y,0}$; that is, it reduces the parametrization of nilpotent orbits from a non-archimedean local field to its residue field via Bruhat-Tits theory.

\begin{theorem}\label{De1}\cite[Theorem 5.6.1]{De02b} Assume $\mathrm{char}(k)\gg0$. Let $e\in\Lg$ be nilpotent. Then there exists $y\in\CA_{\bS}$ and $\bar{e}\in\Lg_{y,0}$ nilpotent such that\hp
(i)$\;\;$ The $G$-conjugacy class of $e$ meets $\bar{e}+\Lg_{y,>0}$.\qp
(ii)$\;$ For any nilpotent element $e'$ in $\bar{e}+\Lg_{y,>0}$, either $e'$ is conjugate to $e$, or $\dim\Ad(G)(e')>\dim\Ad(G)(e)$.\qp
\noindent We may always choose $y$ to lie inside the closure of the alcove containing $x$.
\end{theorem}\hp

DeBacker's result actually says $(y,\bar{e})$ can be made unique in an appropriate sense, giving a parametrization of nilpotent orbits in $\Lg$; we will not need this stronger formulation. The second result, which was built on the first, is an enhancement of the so-called Shalika germ expansion. For $\gamma\in\Lg$ let $I_{\gamma}\in C_c^{\infty}(\Lg)^*$ be the orbital integral over $\Ad(G)\gamma$ (see the first two paragraphs in Section \ref{secnorm} for a discussion and our normalization).

\begin{theorem}\label{De2}\cite[Theorem 2.1.5(3)]{De02a} Let $\{e_1,e_2,...\}$ be the set of nilpotent orbits in $\Lg$.
Fix $r\in\R$. For any $\gamma$ with $\depth(\gamma)>r$, there exists constants $\Gamma_{e_j}(\gamma)\in\Q$, one of each nilpotent orbits, such that
for any $f\in C_c^{\infty}(\Lg)$ locally constant by $\Lg_{y,>r}$ for some $y\in\CA_{\bS}$, i.e. $f(\gamma_1+\gamma_2)=f(\gamma_1)$ for any $\gamma_1\in\Lg$, $\gamma_2\in\Lg_{y,>r}$, we have
\[I_{\gamma}(f)=\sum_j\Gamma_{e_j}(\gamma)I_{e_j}(f).\]
\end{theorem}
\noindent Note that the constants $\Gamma_{e_j}(\gamma)$ are intrinsic in $\gamma$, so that as a distribution $I_{\gamma}$ is somewhat comparable to $\sum\Gamma_{e_j}(\gamma)I_{e_j}$.\p

We now describe the so-called Bruhat-Tits building $\CB_{\bG}$, which we will not use (except through DeBacker's results) in this article. Bruhat-Tits theory asserts that $\CB_{\bG}$ exists as a polyhedral complex on which $G$ acts, and satisfies
\begin{fact}\label{FactBT}
(i)$\;\;$ $\CB_{\bG}$ is the union of $g.\CA_{\bS}$ for all $g\in G$; each $g.\CA_{\bS}$ is called an apartment.\qp
(ii)$\;$ The parahoric $G_{y,\ge 0}$ acts trivially on $y\in\CA_{\bS}$.\qp
(iii) For any $y'\in\CB_{\bG}$ with $y'=g.y$, $y\in\CA_{\bS}$, the group $G_{y',\ge 0}:=\Ad(g)G_{y,\ge 0}$ depends only on $y'$. We call $G_{y',\ge 0}$ the parahoric subgroup at $y'$. Similarly the Moy-Prasad filtration $G_{y',>0}:=\Ad(g)G_{y,>0}$ and $\Lg_{y',\ge r}:=\Ad(g)G_{y,\ge r}$, etc, and also the reductive quotient $\bG_{y',0}$ are well-defined.\qp
(iv) $G$ acts transitively on the set of alcoves. Hence all Iwahori subgroups are conjugate under $G$.\qp
(v)$\,$ Any compact subgroup of $G$ has a fixed point on $\CB_{\bG}$.
\end{fact}\hp

In particular, from (iii) above we see that we may define the depth (see \ref{depth}) as $\depth(\gamma):=\min\{r\in\R\;|\;\Ad(G)\gamma\in\Lg_{y,\ge r}\text{ for some }y\in\CB_{\bG}\}$. One may similarly replace $\CA_{\bS}$ by $\CB_{\bG}$ in Theorem \ref{De1} and \ref{De2} and relax the $G$-conjugacy. In this sense $\CB_{\bG}$ is more intrinsic, and moreover $\CB_{\bG}$ can be defined to have similar properties when $\bG$ is not split over $k$ but even only defined over $F$; we will nevertheless be content with the easier object $\CA_{\bS}$ in this article.\p

\begin{remark}\label{unram} When we make a finite unramified base change $F'/F$, or equivalently a finite base change $k'/k$ with $F'=k'((t))$, the base change $\bS':=\bS\times\Spec k'$ of $\bS'$ to $k'$ is still a maximal $k'$-split torus of $\bG\times\Spec k'$. One may identify $\CA_{\bS'}\cong X_*(\bS)\otimes\R\cong\CA_{\bS}$. In this sense
we identify the apartments corresponding to $\bS$ for all finite unramified extensions $F'/F$. (It's also possible embed the building for $F$ into that for $F'$. That will be the common practice, but we won't need it.)
\end{remark}

\section{Orbital integrals and normalizations}\label{secnorm}

In this section we discuss orbital integrals and their different normalizations. We will write $v:F^{\times}\ra\Z$ the normalized valuation so that $v(t)=1$. We recall for any smooth variety $X$ over $F$, the set $X(F)$ has a natural structure of an
$F$-analytic manifold, i.e. a topological space equipped with coordinate charts from $\CO^n$ such that coordinate change functions are analytic. It thus makes sense to talk about the tangent space at any point on $X(F)$. For example, $T_e(G)=\Lg$. We note that if $U\subset\Lg$ is a neighborhood of $0$ and $j:U\ra G$ is a coordinate chart with $j(0)=e$ (the identity in $G$) and $dj=\Id|_{\Lg}$, then
\begin{equation}\label{tangent}
j(\Lg_{y,\ge r})=G_{y,\ge r}
\end{equation}
for $r\gg 0$ (depending on $j$). This can be verified from generators of $\Lg_{y,\ge r}$ (resp. $G_{y,\ge r}$), which are either root subalgebra (resp. subgroup), or are in $\Lie\bS$ (resp. $\bS$) for which the statement reduces to $\bG=\mathbb{G}_m$, $\Lg=F$ and $G=F^{\times}$.\p

For any $\gamma\in\Lg$, its orbit $\Ad(G)\gamma$ is open in $(\Ad(\bG)\gamma)(F)$ and is also an $F$-analytic manifold. 
Up to constant there is a unique $G$-invariant measure on $\Ad(G)\gamma$. Let $C_c^{\infty}(\Lg)$ be the space of $\C$-valued functions on $\Lg$ that are compactly supported and locally constant. We denote by $I_{\gamma}(f)$ for the integral of $f$ over $\CO_G(\gamma)$; it is known to always converge \cite{Ra72}, \cite[Theorem 61]{Mc04}. Nevertheless $I_{\gamma}$ depends on the normalization of the measure on $\Ad(G)\gamma$. We now give this normalization\p

By assumption on $\mathrm{char}(k)$ there exists and we fix a $\bG$-invariant bilinear form on $\Lie\bG$. Over $F$ this gives $B(\cdot,\cdot):\Lg\times\Lg\ra F$. From (\ref{MPgrading}) we have
\begin{equation}\label{dual}
\Lg_{y,>0}=\{Y\in\Lg\,|\,B(X,Y)\in tk[[t]],\;\forall X\in\Lg_{y,\ge 0}\}
\end{equation}
for any $y\in\CA_{\bS}$. For any $\gamma\in\Lg$, we have a canonical identification of tangent space
\begin{equation}\label{tangent2}
\iota:T_{\gamma}(\Ad(G)\gamma)\cong\Lg/Z_{\Lg}(\gamma)
\end{equation}
given by the fact that the $F$-points of the (algebraic) tangent space of a smooth variety over $F$ is equal to the (analytic) tangent space of the $F$-points of the variety.\p

We adapt the identification $\iota$. For any lattice $L\subset\Lg/Z_{\Lg}(\gamma)$, let us write $L^*=\{Y\in\Lg/Z_{\Lg}(\gamma)\,|\,B(X,[Y,\gamma])\in tk[[t]],\;\forall X\in L\}$. We assign a measure $m_{\gamma}$ on $\Lg/Z_{\Lg}(\gamma)$ such that
\begin{equation}\label{meam}
m_{\gamma}(L)m_{\gamma}(L^*)=1
\end{equation}
for any $L$. Doing so for every element in $\Ad(G)\gamma$ induces a $G$-invariant measure on this orbit. For any $f\in C_c^{\infty}(\Lg)$, we denote by $I_{\gamma}(f)$ the integral of $f$ over $\Ad(G)\gamma$ with respect to the above measure. Recall that two elements $\gamma_1,\gamma_2\in\Lg$ are in the same stable orbit if $\gamma_2\in\Ad(F^{ur})(\gamma_1)$. Any stable orbit is a finite union of ($G$-)orbits. We then write $I_{\gamma}^{st}(f)$ for the sum of $I_{\gamma'}(f)$ where $\gamma'$ runs over a set of representatives of (the finite set of) the orbits in the stable orbit of $\gamma$.\p

\begin{lemma}\label{dilate}
Let $f\in C_c^{\infty}(\Lg)$ and let $f_{t^{-1}}$ be a dilation of $f$ by $t^{-1}$, i.e. $f_{t^{-1}}(X):=f(tX)$. Then $I_{t^{-1}\gamma}(f_{t^{-1}})=q^{\frac{1}{2}\dim\Ad(G)\gamma}\cdot I_{\gamma}(f)$.
\end{lemma}

\begin{proof} There is a natural identification of the tangent space $T_{t^{-1}\gamma}(\Ad(G)(t^{-1}\gamma))\cong T_{\gamma}(\Ad(G)\gamma)$ by $X\mapsto tX$. However the symplectic form $B(\cdot,[\cdot,t^{-1}\gamma])$ on the first is $t^{-1}$ times the form $B(\cdot,[\cdot,\gamma])$ on the second. The resulting measure thus differ by $|t^{-1}|$ raised to the power of half of the dimension of the symplectic space, namely $q^{\frac{1}{2}\dim\Ad(G)\gamma}$.
\end{proof}\hp

We compare the normalization $I_{\gamma}$ with other common candidates when $\gamma$ is regular semisimple. In \cite[Sec. 2.6]{DK06} the normalization of the invariant measure on $G/Z_G(\gamma)$ is given as follows: there is a $G$-invariant measure on $\Lg$ such that
\[|\Lg_{y,\ge0}|\cdot|\Lg_{y,>0}|=1\]
for any $y$, after which there exists a Haar measure on $G$ such that $|G_{y,>0}|=|\Lg_{y,>0}|$ for any $y$. We note that $\Lg_{y,>0}$ and $\Lg_{y,\ge0}$ play the role of $L$ and $L^*$ in the last paragraph by (\ref{dual}). The
Haar measure on the torus $Z_G(\gamma)$ is chosen with a similar manner: Let $\bT:=Z_{\bG_F}(\gamma)$ be the underlying torus defined over $F$, where $\bG_F:=\bG\times_{\Spec k}\Spec F$. The torus $\bT$ has a connected N\`{e}ron-Raynaud model over $\CO$ (which we still denote by $\bT$). Let $\bT_0$ be the reductive quotient of $\bT\times_{\operatorname{Spec}\CO}\operatorname{Spec}k$, $T_{\ge 0}:=\bT(\CO)$ and $T_{>0}:=\ker(T_{\ge 0}\sra\bT_0(k))$ (In the general Bruhat-Tits theory as in \cite{Tits}, $T_{\ge 0}$ is the unique parahoric subgroup of $T:=\bT(F)=Z_G(\gamma)$.). The lattices $\Lt_{\ge0},\Lt_{>0}$ are defined similarly via replacing above algebraic groups by their Lie algebras. Having $T$ (resp. $\Lt$) in the place of $G$ (resp. $\Lg$) defines a Haar measure on $Z_G(\gamma)=\bT(F)$. We now take the quotient measure on $G/Z_G(\gamma)$, and denote by $I_{\gamma}^{DK}(f)$ the orbital integral over $\Ad(G)\gamma\cong G/Z_G(\gamma)$ with the quotient measure. Comparing two definitions we have
\begin{equation}\label{norm1}
I_{\gamma}(f)=q^{-v(D(\gamma))/2}\cdot I_{\gamma}^{DK}(f),
\end{equation}
where $D(\gamma)\in F^{\times}$ is the determinant of $\mathrm{ad}(\gamma):\Lg/Z_{\Lg}(\gamma)\ra\Lg/Z_{\Lg}(\gamma)$ and recall $v(D(\gamma))$ is its normalized valuation of $D(\gamma)$.\p

Another normalization, used for example in \cite{GKM04} and \cite{Ngo10}, assigns the Haar measure on $G$ and $Z_G(\gamma)$ by requiring a preferred parahoric (most commonly a hyperspecial) subgroup of $G$ and the unique parahoric subgroup of the torus $Z_G(\gamma)$ to have measure $1$. Suppose $G_{y,\ge0}$ is the preferred parahoric of $G$ and $T_{\ge0}$ is the unique parahoric subgroup of the torus $Z_G(\gamma)$. As before let $\bG_{y,0}$ and $\bT_0$ be their corresponding reductive quotients. If we denote by $I_{\gamma}^{GKM}(f)$ the orbital integral defined by this normalization, we have
\begin{equation}\label{norm2}
I_{\gamma}^{DK}(f)=\frac{|G_{y,0}|}{|T_0|}q^{(-\dim\bG_{y,0}+\dim\bT_0)/2}\cdot I_{\gamma}^{GKM}(f).
\end{equation}
This can be quickly seen as follows: In the normalization of $I_{\gamma}^{DK}(f)$, we have the Haar measure on $\Lg$ is such that $1=|\Lg_{g,\ge0}|\cdot|\Lg_{g,>0}|=q^{\dim\bG_{y,0}}\cdot|\Lg_{g,>0}|^2$. Thus $|G_{y,>0}|=|\Lg_{y,>0}|=q^{-\dim\bG_{y,0}/2}$ and $|G_{y,\ge0}|=|G_{y,0}|\cdot q^{-\dim\bG_{y,0}/2}$. A similar factor for $T$ appears on the other side of the quotient, and thus the formula follows.\p

For an open compact subset $V\subset\Lg$ we will denote by
We will denote by $1_V$ the function on $\Lg$ that takes the value $1$ on $V$ and $0$ elsewhere. In this article, we will in fact not only work with a single $k$, but with all finite extensions at the same time. Denote by $q=|k|$ (which depends on $k$). We will frequently write statements of the form
\[I_{\gamma}(1_{\Lg_{x,\ge 0}})=O(1).\]
Such a statement will mean that there exists a constant $C$, depending possibly on $\gamma$ but independent of the chosen finite extension of $k$, such that for any $k$, the orbital integral $I_{\gamma}(1_{\Lg_{x,\ge1}})$ bounded by $C$. Here $\gamma$ is realized in $(\Lie\bG)(k((t)))$, both depending on $k$. And for the definition of $\Lg_{x,\ge0}$ we refer to Remark \ref{unram} as $k$ varies. Likewise,
\[I_{\gamma}^{st}(1_{\Lg_{x,\ge0}})=1+O(q^{-1/2})\]
will mean that there exists a constant $C$, such that for any finite extension $k$ we have $1-Cq^{-1/2}\le I^{st}_{\gamma}(1_{\Lg_{x,\ge0}})\le 1+Cq^{-1/2}$, where $q=|k|$ depends on $k$, and $\gamma$, $I_{\gamma}$ and $1_{\Lg_{x,\ge0}}$ vary with $k$ in the above sense.\p

\section{Geometric preparation}\label{secgeom}

In this section, we introduce the geometric tools used in this article. Readers familiar with affine Springer fibers and comfortable with their various generalizations may skip this section except for a look at Theorem \ref{GKM}.
The geometric starting point of this work is the following: Suppose there is a variety $X$ over $k\cong\mathbb{F}_q$ and we want to count the number of its irreducible components. Suppose $X\times_{\Spec k}\Spec\bar{k}$ (which we'll  later denote by $X\times\Spec\bar{k}$ for convenience) has $C_k$ components of the top dimension $d=\dim X$ that are stabilized by $\Gal(\bar{k}/k)$. Here we write $C_k$ to emphasize that the number varies under base change. In the language introduced at the end of Section \ref{secnorm}, we have
\begin{equation}\label{basicest}
|X(k)|=C_k\cdot q^d+O(q^{k-\frac{1}{2}})=C_k\cdot q^d(1+O(q^{-1/2})).
\end{equation}
And moreover (\ref{basicest}) is equivalent to the assertion of components. To see this, let $\Frob$ be the geometric Frobenius which is the automorphism ($t^q\mapsto t$) on $\bar{k}$. Recall that the Grothendieck-Lefschetz Theorem says
\begin{equation}\label{GL}
|X(k)|=\Tr(\Frob:H_c^*(X\times\Spec\bar{k})).
\end{equation}
Here the cohomology is always the $\ell$-adic \'{e}tale cohomology, and we adapt the convention that when taking trace on a cohomology we have an extra negative sign for the odd-degree part. Each of the eigenvalues of $\Frob$ acting on the cohomology has absolute value $q^{w/2}$, where $w$ is called the {\bf weight}. The dual space of the top-degree subspace $H_c^{2d}(-)$ has a basis indexed by the top-dimensional components of $X\times\Spec\bar{k}$, and the action of $\Frob$ on $H_c^{2d}(-)$ is $q^d$ times the dual action of $\Frob\in\Gal(\bar{k}/k)$ on the set of components. The weights on the rest (i.e. $H_c^{<2d}(-)$) are strictly less than $2d$. This implies (\ref{basicest}).\p

Next we introduce affine Springer fibers. We refer the readers to \cite[\S2]{Yu16} for detailed constructions. Fix $\gamma\in\Lg$ regular semisimple. Let $I:=G_{x,\ge 0}\subset G=\bG(F)$ be an Iwahori subgroup introduced in Section \ref{secBT}. The affine flag variety is an ind-variety\footnote{The usual affine flag varieties are, in fact, ind-schemes that are not ind-reduced unless $\bG$ is semisimple. Nevertheless, for the purpose of our article we may and shall take the reduced structure.} $\CX$ with a natural identification $\CX(k)=G/I$ and the same for any finite $k'/k$ with $G$ replaced by $\bG(F')$, $F'=k'((t))$ and $I=G_{x,\ge 0}$ replaced in the manner of Remark \ref{unram}. The Iwahori affine Springer fiber $\CX_{\gamma}$ is a closed sub-ind-variety of $\CX$ such that
\begin{equation}\label{ASF}
\CX_{\gamma}(k)=\{g\in G/I\;|\;\Ad(g^{-1})\gamma\in\Lie I\},
\end{equation}
where $\Lie I:=\Lg_{x,\ge 0}$ (and similarly for $k'/k$ finite). With the assumption that $\gamma$ is regular semisimple, $\CX_{\gamma}$ is locally of finite type over $k$ \cite[\S3]{KL88} (see also \cite[Thm. 2.5.2]{Yu16}), but might have infinitely many irreducible components.

\begin{lemma}\label{equi} The ind-variety $\CX_{\gamma}$ is equi-dimensional.
\end{lemma}

\begin{proof} In \cite[Prop. 4.1]{KL88} the equi-dimensionality is proved under the assumption that $\bG$ is simply connected. Let $\bG^{sc}$ be the simply connected cover of the derived group $\bG^{der}$ of $\bG$. By \cite[Prop. 6.6]{PR08} $\CX$ has the same identity component as $\CX^{der}$, the affine flag variety for $\bG^{der}$. By the paragraph after (6.11) in {\it ibid.}, the identity component of $\CX^{der}$ is isomorphic to $\CX^{sc}$, the affine flag variety for $\bG^{sc}$. Thus $\CX$ is a disjoint union of translates of $\CX^{sc}$; say
\begin{equation}\label{redtosc}
\CX=\bigsqcup g_i\cdot\CX^{sc},\;g_i\in\bG(F^{ur}).
\end{equation}
Our assumption on $\mathrm{char}(k)$ ensures $\Lie\bG=\Lie\bG^{sc}\times\Lie Z(\bG)^o$, with which we may decompose $\gamma=\gamma^{sc}+\gamma^z$. If $\gamma^z\not\in(\Lie Z(\bG)^o)(\CO)$ then $\CX_{\gamma}=\emptyset$ and there is nothing to prove. Otherwise, one checks from (\ref{ASF}) that we have $\CX_{\gamma}=\bigsqcup\CX^{sc}_{\gamma_i}$ where $\gamma_i:=\Ad(g_i)^{-1}(\gamma^{sc})$, $g_i$ are as in (\ref{redtosc}), and each $\CX^{sc}_{\gamma_i}$ is the Iwahori affine Springer fiber for $\bG^{sc}$ and $\gamma^{sc}$.
Since each $\CX^{sc}_{\gamma_i}$ is equi-dimensional, it suffices to prove that they all have the same dimension. This is implied by the dimension formula of Bezrukavnikov (\ref{Bezr}) and note that conjugating $\gamma^{sc}$ by $g_i^{-1}$ in the formula doesn't change the outcome.
\end{proof}\hp


To study the components of $\CX_{\gamma}$ in the spirit of \ref{basicest}, we consider the centralizer $\bT:=Z_{\bG_F}(\gamma)$. This is a torus defined over $F$. It's possible to consider this as a pro-group scheme over $k$. That is, there is a natural projective limit $\CT$ of group schemes with natural identifications $\CT(k')=\bT(F')$ for any finite extension $k'\supset k$ and $F'=k'((t))$. Let us write $T=\CT(k)=\bT(F)=Z_G(\gamma)$. The group $\CT$ acts on $\CX$ by left translation. From (\ref{ASF}) one easily sees that the left translation action of $\CT(k)=T$ on $\CX(k)=G/I$ preserves $\CX_{\gamma}(k)$ and the same is true for any $k'/k$. Hence $\CT$ also acts on $\CX_{\gamma}$.
In fact, this $\CT$-action on $\CX_{\gamma}$ (resp. $T$-action on $\CX_{\gamma}(k)$) factors through a finite-dimensional (resp. finite) quotient. The ind-variety $\CX_{\gamma}$ has only finitely many $\CT$-orbits of components.\p

To apply (\ref{basicest}) in terms of $\CT$-orbits, there is still the issue that a $\CT$-orbit in $\CX_{\gamma}(\bar{k})$ defined over $k$ is not necessarily a $T$-orbit on $\CX_{\gamma}(k)$. Suppose there exists a $h'\in\bG(F^{ur})$ such that $\gamma':=\Ad(g')\gamma\in\Lie I$. Then $h'\in\bG(F')$ for some $F'=k'((t))$, $k'/k$ finite. The image of $h'$ in $\CX(k')$ then lies in $\CX_{\gamma}(k')$, and its $\CT$-orbit is defined over $k$. This suggests that when one considers $\CT$-orbits of $k$-points and in the same spirit the $\CT$-fixed part of the cohomology, they should be compared with stable orbital integrals. Indeed, this is the $\kappa=1$ special case of \cite[Theorem 15.8]{GKM04}.

\begin{theorem}\label{GKM} (Goresky-Kottwitz-MacPherson) Let $\gamma\in\Lg$ be regular semisimple. Then
\[I_{\gamma}^{st}(1_{\Lg_{x,\ge 0}})=\frac{|G_{x,0}|}{|T_0|}\cdot q^{(-v(D(\gamma))-\dim\bG_{x,0}+\dim\bT_0)/2}\cdot\mathrm{Tr}(\mathrm{Frob}\,;\,H^*(\CX_{\gamma}\times\Spec\bar{k})^{\CT}).\]
\end{theorem}\hp

Note that the extra factors before the Frobenius trace comes from (\ref{norm1}) and (\ref{norm2}). In \cite[3.4.11]{Yu16} a simpler proof of the $\kappa=1$ case is also given except that we need to replace \cite[3.2.6]{Yu16} by (\ref{GKM}); we will refer to this proof later. Let $d=\dim\CX_{\gamma}$. Thanks to Lemma \ref{equi}, $H^{2d}(\CX_{\gamma}\times\Spec\bar{k})^{\CT}$ (or equivalently, its dual space) has a basis indexed by $\CT$-orbits of irreducible geometric components. Let $C_k$ be the number of such orbits that are stabilized by $\Gal(\bar{k}/k)$. Then as a consequence of Theorem \ref{GKM} we have
\begin{equation}\label{GKMe}
I_{\gamma}^{st}(1_{\Lg_{x,\ge 0}})=C_k\cdot q^{d+(-v(D(\gamma))+\dim\bG_{x,0}-\dim\bT_0)/2}\cdot (1+O(q^{-1/2})).
\end{equation}
Here we also apply (\ref{basicest}) to $|G_{x,0}|$ and $|T_0|$ and use that both $\bG_{x,0}$ and $\bT_0$ are irreducible.
\p

Analogues of Theorem \ref{GKM} and (\ref{GKMe}) also hold for various generalized affine Springer fibers. Fix $y\in\CA_{\bS}$. To begin with one may replace $I=G_{x,\ge0}$ by any parahoric subgroup $G_{y,\ge 0}$, and the affine Springer fiber for $y$ is an ind-variety $\CX_{y,\gamma}$ with which the analogue of (\ref{ASF}) holds, that is
\begin{equation}\label{ASF2}
\CX_{y,\gamma}(k)=\{g\in G/G_{y,\ge 0}\;|\;\Ad(g^{-1})\gamma\in\Lg_{y,\ge 0}\},
\end{equation}
and likewise for any $k'/k$ finite. In this case, Theorem \ref{GKM} and (\ref{GKMe}) become
\begin{equation}\label{GKM2}
I_{\gamma}^{st}(1_{\Lg_{y,\ge 0}})=\frac{|G_{y,0}|}{|T_0|}\cdot q^{(-v(D(\gamma))-\dim\bG_{y,0}+\dim\bT_0)/2}\cdot\mathrm{Tr}(\mathrm{Frob}\,;\,H^*(\CX_{y,\gamma}\times\Spec\bar{k})^{\CT}).
\end{equation}
and consequently 
\begin{equation}\label{GKMe2}
I_{\gamma}^{st}(1_{\Lg_{y,\ge 0}})=C_k\cdot q^{\dim\CX_{y,\gamma}+(-v(D(\gamma))+\dim\bG_{y,0}-\dim\bT_0)/2}\cdot (1+O(q^{-1/2}))
\end{equation}
where $C_k$ is the number of $\Frob$-invariant $\CT$-orbits of top-dimensional geometric components of $\CX_{y,\gamma}$. Equation (\ref{GKM2}) is also a special case of \cite[Theorem 15.8]{GKM04}, for which a simpler proof can be given as in \cite[3.4.11]{Yu16} as long as we replace \cite[3.2.6]{Yu16} by (\ref{ASF2}).\p

We may further generalize (\ref{GKM2}) and (\ref{GKMe2}). Keep $y\in\CA_{\bS}$ and fix an element $\bar{e}\in\Lg_{y,0}$. Consider the locally closed subvariety $\Ad(\bG_{y,0})\bar{e}\in\Lie\bG_{y,0}$. Let $\Lg_{y,\ge0}^{(\bar{e})}\subset\Lg_{y,\ge0}$ be the preimage of $(\Ad(\bG_{y,0})\bar{e})(k)\subset\Lg_{y,0}$. There is a generalized affine Springer fiber $\CX_{y,\bar{e},\gamma}$, which is an locally closed sub-ind-variety of $\CX_{y,\gamma}$, such that (and likewise for $k'/k$ finite)
\begin{equation}\label{ASF3}
\CX_{y,\bar{e},\gamma}(k)=\{g\in G/G_{y,\ge 0}\;|\;\Ad(g^{-1})\gamma\in\Lg_{y,\ge 0}^{(\bar{e})}\},
\end{equation}
Replacing \cite[3.2.6]{Yu16} by (\ref{ASF3}), the proof in \cite[3.2.11]{Yu16} gives the following variant of Theorem \ref{GKM}
\begin{equation}\label{GKM3}
I_{\gamma}^{st}(1_{\Lg_{y,\ge 0}^{(\bar{e})}})=\frac{|G_{y,0}|}{|T_0|}\cdot q^{(-v(D(\gamma))-\dim\bG_{y,0}+\dim\bT_0)/2}\cdot\mathrm{Tr}(\mathrm{Frob}\,;\,H_c^*(\CX_{y,\bar{e},\gamma}\times\Spec\bar{k})^{\CT}).
\end{equation}
and consequently
\begin{equation}\label{GKMe3}
I_{\gamma}^{st}(1_{\Lg_{y,\ge 0}^{(\bar{e})}})=C_k\cdot q^{\dim\CX_{y,\bar{e},\gamma}+(-v(D(\gamma))+\dim\bG_{y,0}-\dim\bT_0)/2}\cdot (1+O(q^{-1/2}))
\end{equation}
where $C_k$ is the number of $\Frob$-invariant $\CT$-orbits of top-dimensional geometric components of $\CX_{y,\bar{e},\gamma}$. Equation (\ref{GKMe3}) will be used in Section \ref{secSha} in proving (\ref{A6}).\p

Now suppose in (\ref{ASF2}) and (\ref{ASF3}) $y$ lies in the closure of the alcove containing $x$ (see the paragraph before (\ref{depth})). Then there exists a natural map
\begin{equation}\label{mapSpr}
\CX_{\gamma}\ra\CX_{y,\gamma}.
\end{equation}
On the level of $k$-points (and likewise for $k'$-points) it comes from the natural quotient map $G/G_{x,\ge 0}\sra G/G_{y,\ge 0}$. Above any $\bar{k}$-point in $\CX_{y,\bar{e},\gamma}\subset\CX_{y,\gamma}$, the fiber is by definition isomorphic to
\begin{equation}\label{SprF}
\CB_{\bar{e}}:=\{g\in\bG_{y,0}/\mathbf{B}\;|\;\Ad(g^{-1})(\bar{e})\in\Lie\mathbf{B} \}
\end{equation}
where $\mathbf{B}\subset\bG_{y,0}$ is any Borel subgroup. The variety $\CB_{\bar{e}}$ described in (\ref{SprF}) is what is called a {\bf Grothendieck-Springer fiber}, or just {\bf Springer fiber} when $\bar{e}$ is nilpotent. It has a well-known dimension formula
\begin{equation}\label{Sprdim}
\dim\CB_{\bar{e}}=\frac{1}{2}(\dim Z_{\bG_{y,0}}(\bar{e})-\rank\bG_{y,0}).
\end{equation}
\p




\section{Dimension formulas}\label{secdim}

As in Section \ref{secBT} we fix an alcove $x$, so that we have an Iwahori subgroup $I:=G_{x,\ge0}$. Fix $\gamma\in\Lg=(\Lie\bG)(F)$ regular semisimple as in (\ref{ASF}). Recall that we have (\ref{GKMe}):
\[I_{\gamma}^{st}(1_{\Lg_{x,\ge 0}})=C_k\cdot q^{\dim\CX_{\gamma}+(-v(D(\gamma))+\dim\bG_{x,0}-\dim\bT_0)/2}\cdot (1+O(q^{-1/2})).\]
where $C_k$ is the number of $\CT$-orbits of geometric components of $\CX_{\gamma}\times\Spec\bar{k}$ that are stabilized by $\Gal(\bar{k}/k)$. Note that $\bG_{x,0}=\bS$ and $\dim\bG_{x,0}=\rank\bG$. On the other hand, it was conjectured in \cite[\S0]{KL88} and proved by Bezrukavnikov \cite{Be96} that 
\begin{equation}\label{Bezr}
\dim\CX_{\gamma}=\frac{1}{2}(v(D(\gamma))-\rank\bG+\dim\bT_0)
\end{equation}
when $\CX_{\gamma}$ is non-empty\footnote{\label{rem}The result in \cite{Be96} was stated for $\gamma$ topologically nilpotent, but one can reduce compact $\gamma$ to the topologically nilpotent case as in \cite[\S5]{KL88}.}. Since a stable orbital integral is a sum of orbital integrals (all non-negative in the following), this implies

\begin{corollary}\label{dim}For $\gamma\in\Lg$ regular semisimple, we have $I_{\gamma}(1_{\Lg_{x,\ge 0}})=O(1)$.
\end{corollary}

Nevertheless, in Section \ref{secRR} we will show that the same result holds for nilpotent orbits:

\begin{proposition}\label{comb} For $e\in\Lg$ nilpotent, we have $I_e(1_{\Lg_{x,\ge 0}})=O(1)$.
\end{proposition}\hp

It will be nice to have a geometric proof for Proposition \ref{comb}, and/or a $p$-adic analytic proof for Corollary \ref{dim}. We'd also like to make the following conjecture.\p

\begin{conjecture} For any $\gamma\in\Lg$, we have $I_{\gamma}(1_{\Lg_{x,\ge 0}})=O(1)$.
\end{conjecture}\hp

\begin{remark} Suppose one normalizes orbital integrals over our finite field $k$ (which becomes normalized sums) as follows: Say $\bH$ is a reductive group over $k$. Let $\bar{\gamma}\in\mathfrak{h}=(\Lie\bH)(k)$. For any function $f$ on $\mathfrak{h}$, define $I_{\bar{\gamma}}(f)$ to be the sum of the function along $({}^{\bH}\bar{\gamma})(k)$, multiplied by $q^{-\frac{1}{2}\dim{}^{\bH}\bar{\gamma}}$. Note that tangent spaces of the orbit are vector spaces of dimension $\dim{}^{\bH}\bar{\gamma}$ over $k$. Hence this agrees with our normalization in Section \ref{secnorm}. It turns out that this $I_{\bar{\gamma}}$ is a distribution given by the pure perverse sheaf $\underline{\bar{\Q}_{\ell}}[\dim{}^{\bH}\bar{\gamma}](-\frac{1}{2}\dim{}^{\bH}\bar{\gamma})$ on ${}^{\bH}\bar{\gamma}$.\hp

On the other hand, a direct calculation shows that $I_{\bar{\gamma}}(1_{\mathfrak{b}})=O(1)$ is bound given by the well-known dimension formula (\ref{Sprdim}) for the Grothendieck-Springer fibers. Therefore, Bezrukavnikov's result may be realized as the exact affine generalization of the dimension formula for the Grothendieck-Springer fibers.
\end{remark}\p

\section{Shalika germ expansion}\label{secSha}

We now assume $\gamma=t\gamma_0\in\Lg$ is regular semisimple with $\gamma_0$ topologically nilpotent. In other words $\gamma$ has depth $>1$. We again write $\bT:=Z_{\bG}(\gamma)$, and denote by $W$ the Weyl group of $\bG$. Recall that we have fixed an $x\in\CA_{\bS}$ that lies in an alcove. Our goal is to prove the following stronger version of (\ref{dim})
\begin{equation}\label{maineq0}
I_{\gamma}^{st}(1_{\Lg_{x,\ge 0}})=|W|+O(q^{-1/2}).
\end{equation}

\begin{lemma} Equation (\ref{maineq0}) implies Theorem \ref{main}.
\end{lemma}

\begin{proof} One plugs (\ref{Bezr}) and (\ref{maineq0}) into (\ref{GKMe}) and concludes that the number of $\CT$-orbits of components of $\CX_{\gamma}$ that are stabilized by $\Gal(\bar{k}/k)$ is equal to the order of the Weyl group. Since holds for any finite base change of $k$, Theorem \ref{main} is proved.
\end{proof}\hp

The rest of this section is devoted to the proof of (\ref{maineq0}). Recall $\Lg_{x,\ge -1}=t^{-1}\Lg_{x,\ge 0}$. Thanks to Lemma \ref{dilate}, equation (\ref{maineq0}) is equivalent to
\begin{equation}\label{maineq}
I_{\gamma_0}^{st}(1_{\Lg_{x,\ge -1}})= q^{\frac{1}{2}\dim\Ad(G)\gamma}\cdot(|W|+O(q^{-1/2})).
\end{equation}\p

We note that as stable orbits in the adjoint group $\bG_{ad}$ are the same as in the original group, (\ref{maineq0}) and (\ref{maineq}) can be passed to $\bG_{ad}$. We thus may and shall assume $\bG$ is adjoint. Let $\{e_1,e_2,...\}$ be a set of representatives for the set of nilpotent orbits in $\Lg$. We assume that $e_1,e_2,...$ are ordered so that their orbits have decreasing dimensions. As $\bG$ is assumed to be adjoint, there is only one regular nilpotent orbit $e_1$.\p

By Theorem \ref{De1} of DeBacker, for each $e_i$ there exists some point $x_i\in\CA_{\bS}$ and $\bar{e}_i\in\Lg_{x_i,0}$, such that the orbit of $e_i$ is the (unique) smallest orbit that meets $\bar{e}_i+\Lg_{x_i,>0}$. We assume $x_i$ are chosen so that $x_i$ is in the closure of the alcove containing $x$. Write $d_i=\dim\Ad(\bG_{x_i,0})\bar{e}_i$ the dimension of the $\bG_{x_i,0}$-orbit of $\bar{e}_i$. Let $f_i\in C_c^{\infty}(\Lg)$ be the function that takes the value $q^{d_i/2}$ on $\bar{e}_i+\Lg_{x_i,>0}$ and $0$ elsewhere. By Theorem \ref{De2} (again of DeBacker) with $r=0$ and $\gamma_0$ in the place of $\gamma$, we have Shalika germ expansions
\[\def\arraystretch{1.4}\begin{array}{lll}
I_{\gamma_0}^{st}(1_{\Lg_{x,\ge -1}})&=&\dsp\sum_j\Gamma_{e_j}^{st}(\gamma_0)I_{e_j}(1_{\Lg_{x,\ge -1}}),\text{ and}\\
I_{\gamma_0}^{st}(f_i)&=&\dsp\sum_j\Gamma_{e_j}^{st}(\gamma_0)I_{e_j}(f_i),\;\;i=1,2,...
\end{array}\]
where, just like $I_{\gamma_0}^{st}$, we define $\Gamma_{e_j}^{st}(\gamma_0)$ to be the sum of $\Gamma_{e_j}(\gamma')$ where $\gamma'$ runs over a set of representatives of (the set of) the orbits in the stable orbit of $\gamma_0$. Combining the two equations above gives
\begin{equation}\label{21b}I_{\gamma_0}^{st}(1_{\Lg_{x,\ge -1}})=\matr{I_{e_1}(1_{\Lg_{x,\ge -1}})&\!\!...&\!\!...\;}
\matr{I_{e_1}(f_1)&0&...&0\\I_{e_1}(f_2)&I_{e_2}(f_2)&0&0\\I_{e_1}(f_3)&I_{e_2}(f_3)&I_{e_3}(f_3)&0\\&...&&...}^{-1}
\matr{I_{\gamma_0}^{st}(f_1)\\I_{\gamma_0}^{st}(f_2)\\I_{\gamma_0}^{st}(f_3)\\...}.
\end{equation}\hp

Note that we use $I_{e_j}(f_i)=0$ for $j>i$, as by definition the support of $f_i$ will be disjoint from the smaller orbit (or of equal dimension but distinct) $e_j$. We have
\begin{lemma} Equation (\ref{maineq}) and thus (\ref{maineq0}) follow from the following six equations/estimates:
\begin{align}
\label{A1}I_{e_1}(1_{\Lg_{x,\ge -1}})&=|W|\cdot q^{\frac{1}{2}\dim\Ad(G)e_1}.\\
\label{A2}I_{e_i}(1_{\Lg_{x,\ge -1}})&=O(q^{\frac{1}{2}\dim\Ad(G)e_i}).\\
\label{A3}I_{e_i}(f_i)&=1.\\
\label{A4}I_{e_j}(f_i)&=O(1).\\
\label{A5}I_{\gamma_0}^{st}(f_1)&=1+O(q^{-1/2}).\\
\label{A6}I_{\gamma_0}^{st}(f_i)&=O(1).
\end{align}
\end{lemma}

\begin{proof} In (\ref{A1}) we note that $\dim\Ad(G)e_1=\dim^G\gamma$ as both $e_1$ and $\gamma$ are regular elements in $\Lg$. Then putting (\ref{A1})$\sim$(\ref{A6}) into (\ref{21b}) gives
\[\begin{array}{ll}
&I_{\gamma_0}^{st}(1_{\Lg_{x,\ge -1}})\\
=& q^{\frac{1}{2}\dim\Ad(G)\gamma}\cdot\matr{|W|&O(q^{-1/2})&O(q^{-1/2})&\!\!...\;}
\matr{1&0&...&0\\O(1)&1&0&0\\O(1)&O(1)&1&0\\&...&&...}^{-1}\matr{1+O(q^{-1/2})\\O(1)\\O(1)\\...}\\
=&q^{\frac{1}{2}\dim\Ad(G)\gamma}\cdot(|W|+O(q^{-1/2}))
\end{array}\]
\vskip-.5cm
\end{proof}\hp

It remains to prove (\ref{A1}) - (\ref{A6}). Among them equations (\ref{A1}), (\ref{A2}) and (\ref{A4}) are about nilpotent orbital integrals and are similar to Proposition \ref{comb}. We leave their proofs to Section \ref{secRR}. In the rest of this section, we prove (\ref{A3}), (\ref{A5}) and (\ref{A6}).\p

\begin{proof}[Proof of (\ref{A3})]
The intersection of $\bar{e}_i+\Lg_{x_i,>0}$ and $\Ad(G)(e_i)$ is a neighborhood of $e_i$. This neighborhood corresponds under a coordinate chart to the image of $\Lg_{x_i,>0}$ in the tangent space $\Lg/Z_{\Lg}(e_i)\cong T_{e_i}(\Ad(G)e_i)$. From (\ref{dual}) and (\ref{meam}) one checks that the image of $\Lg_{x_i,>0}$ has measure $q^{-d_i/2}$. This gives (\ref{A3}).
\end{proof}\hp

\begin{proof}[Proof of (\ref{A6})]
We now prove (\ref{A6}), and later comment on how the proof of (\ref{A5}) follows from carefully inspecting the proof in the case $i=1$. We introduce two auxiliary functions $f_i'$, $f_i^*$ to be compared with $f_i$. We list them together as
\[
\begin{array}{lll}
f_i&:=&\dsp q^{d_i/2}\cdot 1_{\bar{e}_i+\Lg_{x_i,>0}}\\ f_i'&:=&\dsp\frac{q^{d_i/2}}{|\Ad(G_{x_i,0})\bar{e}_i|}\cdot1_{\Ad(G_{x_i,0})\bar{e}_i+\Lg_{x_i,>0}}\\
f_i^*&:=&\dsp q^{-d_i/2}\cdot 1_{(\Ad(\bG_{x_i,0})\bar{e}_i)(k)+\Lg_{x_i,>0}}
\end{array}
\]
The first auxiliary function $f_i$ is an averaged version of $f_i$. That is\hp

\begin{lemma}\label{L12o} We have $I_{\gamma_0}^{st}(f_i)=I_{\gamma_0}^{st}(f_i')$.
\end{lemma}
\begin{proof} For $g\in G$, $f\in C_c^{\infty}(\Lg)$ write $\Ad(g)(f)$ the function with $(\Ad(g)(f))(x)=f(\Ad(g^{-1})x)$. We have $I_{\gamma_0}(\Ad(g)f)=I_{\Ad(g)^{-1}\gamma_0}(f)=I_{\gamma_0}(f)$ and thus $I_{\gamma_0}^{st}(\Ad(g)f)=I_{\gamma_0}^{st}(f)$. Let $g_1,...,g_N$ be set of lifts of $G_{x_i,0}$ in $G_{x_i,\ge 0}$, where $N=|G_{x_i,0}|$. The lemma then follows from that $f_i'=\frac{1}{N}\sum_{j=1}^N\Ad(g_j)f_i$.
\end{proof}\hp

\begin{lemma}\label{L12a} We have
$I_{\gamma_0}^{st}(f_i')/I_{\gamma_0}^{st}(f_i^*)=O(1)$.
\end{lemma}

\begin{proof}
Note that the support of $f_i^*$ contains that of $f_i'$; $\Ad(G_{x_i,0})\bar{e}_i=\Ad((\bG_{x_i,0})(k))\bar{e}_i\subset(\Ad(\bG_{x_i,0})\bar{e}_i)(k)$. Hence to prove the lemma, it suffices to bound the quotient of their values, which is
\begin{equation}\label{L12b}
\frac{q^{d_i}}{|\Ad(G_{x_i,0})\bar{e}_i|}=\frac{q^{\dim\bG_{x_i,0}}}{|G_{x_i,0}|}\cdot\frac{|Z_{G_{x_i,0}}(\bar{e}_i)|}{q^{\dim Z_{\bG_{x_i,0}}(\bar{e}_i)}}=|(\pi_0(Z_{\bG_{x_i,0}}(\bar{e}_i)))(k)|+O(q^{-1/2}).
\end{equation}
\vskip-.8cm
\end{proof}\hp

With the two lemmas, (\ref{A6}) follows from the statement that $I_{\gamma_0}^{st}(f_i^*)=O(1)$. 
Now $I_{\gamma_0}^{st}(f_i^*)$ can be computed using the cohomology of the generalized affine Springer fiber $\CX_{x_i,\bar{e}_i,\gamma_0}\subset\CX_{x_i,\gamma_0}$ introduced in (\ref{ASF3}) (resp. (\ref{ASF2})) with $y$ (resp. $\bar{e}$ and $\gamma$) replaced by $x_i$ (resp. $\bar{e}_i$ and $\gamma_0$). That is, $f_i^*$ is equal to $q^{-d_i/2}$ times $1_{\Lg_{y,\ge 0}^{(\bar{e}_i)}}$ where the latter appears in the LHS of (\ref{GKM3}) and (\ref{GKMe3}). As we assumed that $x_i$ lies in the closure of the alcove containing $x$, the Iwahori affine Springer fiber $\CX_{\gamma_0}=\CX_{x,\gamma_0}$ has a natural map (\ref{mapSpr}) to $\CX_{x_i,\gamma_0}$. A geometric fiber of this map above $\CX_{x_i,\bar{e}_i,\gamma_0}$ has dimension (see (\ref{Sprdim})) $(\dim Z_{\bG_{x_i,0}}(\bar{e_i})-\rank\bG)/2$. The dimension of the Iwahori affine Springer fiber is $(v(D(\gamma_0))-\dim\bG_{x,0}+\dim\bT_0)/2$, and therefore
\begin{equation}\label{a6}
\dim\CX_{x_i,\bar{e}_i,\gamma_0}\le(v(D(\gamma_0))-\dim Z_{\bG_{x_i,0}}(\bar{e_i})+\dim\bT_0)/2.
\end{equation}
On the other hand, (\ref{GKMe3}) gives
\begin{equation}\label{a6b}
I_{\gamma_0}^{st}(f_i^*)=C_k\cdot q^{\dim\CX_{x_i,\bar{e}_i,\gamma_0}+(-v(D(\gamma_0))+\dim\bG_{x_i,0}-\dim\bT_0-d_i)/2}\cdot (1+O(q^{-1/2})).
\end{equation}
Since $\dim\bG_{x_i,0}=d_i+\dim Z_{\bG_{x_i,0}}(\bar{e_i})$, combining (\ref{a6}) and (\ref{a6b}) gives $I_{\gamma_0}^{st}(f_i^*)=C_k'+O(q^{-1/2})$, where $C_k'$ is the number of $\CT$-orbits of components of $\CX_{x_i,\bar{e}_i,\gamma_0}\times\Spec\bar{k}$, of dimension equal to the RHS of (\ref{a6}), that are stabilized by $\Gal(\bar{k}/k)$. In particular $I_{\gamma_0}^{st}(f_i^*)=O(1)$ as asserted and (\ref{A6}) follows by Lemma \ref{L12o} and \ref{L12a}.
\end{proof}\hp

\begin{proof}[Proof of (\ref{A5})]
To establish (\ref{A5}) we would like the above $C_k$ to be equal to $1$ when $i=1$, and $|(\pi_0(Z_{\bG_{x_1,0}}(\bar{e}_1)))(k)|=1$ in (\ref{L12b}). For the former, we note that $\CX_{x_1,\bar{e}_1,\gamma_0}$ is (by definition) the regular locus denoted $O$ in \cite{Be96}. Part (a) of the main proposition in {\it ibid.} asserts that the locus $\CT$ acts transitively on $O$, and part (b) asserts that $O$ has the expected dimension as the RHS of (\ref{a6}), thus $C_k=1$. Next, when $i=1$ we may take $x_1=o$ the hyperspecial vertex and $\bar{e}_1\in\Lg_{o,0}$ a regular nilpotent element. Since $\bG_{o,0}\cong\bG$ is adjoint and $\bar{e}_1$ is the unique regular orbit, $Z_{\bG_{x_1,0}}(\bar{e}_1)$ is irreducible and therefore $|(\pi_0(Z_{\bG_{x_1,0}}(\bar{e}_1)))(k)|=1$.
\end{proof}\hp

\begin{remark} One philosophy underlying the proof is that the number of $\CT$-orbits of the components for $\gamma$ is independent of $\gamma$ as long as $\gamma$ is regular semisimple and $\depth(\gamma)>1$. When $\bT=Z_{\bG}(\gamma)$ is split over $F$, the statement that the number of $\CT$-orbits of components equals to $|W|$ can also be deduced from the method of \cite[Sec. 5]{KL88}. This is in fact how we first realized that $|W|$ is the correct number.
\end{remark}\p

\section{Implementation of Ranga Rao method}\label{secRR}

In \cite{Ra72}, Ranga Rao gave a method for computing nilpotent orbital integrals in order to prove their convergences for compactly supported smooth functions. We review his method here with our normalization. Let $\bS\subset\bG$ be a fixed maximal $k$-split torus as before and write $S=\bS(F)\subset G$. Suppose $e\in\Lg$ is nilpotent, and $\lambda:\mathbb{G}_m\ra\bG$ is a cocharacter associated to $e$ by the Jacobson-Morozov Theorem, so that $\Ad(\lambda(t)e)=t^2e$. By conjugation on $e$ and $\lambda$ we shall assume $\lambda$ has image in $\bS$. We denote by ${}^{\lambda}_i\Lg\subset\Lg$ the subspace on which $\lambda$ acts by $z\mapsto z^i$, so we have $\Lg=\bigoplus{}^{\lambda}_i\Lg$ and $e\in{}^{\lambda}_2\Lg$. Since $\lambda$ has image in $\bS$, this weight space decomposition respect the Moy-Prasad grading (\ref{MPgrading}) and we have for any $y\in\CA_{\bS}$
\[\Lg=\bigoplus_r\bigoplus_i{}^{\lambda}_i\Lg_{y,r},\;{}^{\lambda}_i\Lg_{y,r}:={}^{\lambda}_i\Lg\cap\Lg_{y,r}.\]
We will write ${}^{\lambda}_{\ge j}\Lg:=\bigoplus_{i\ge j}{}^{\lambda}_i\Lg$, ${}^{\lambda}_{<j}\Lg_{y,r}:=\bigoplus_{i<j}{}^{\lambda}_i\Lg_{y,r}$, etc. On the group level we also write $P={}^{\lambda}_{\ge0}G\subset G$ the parabolic subgroup whose Lie algebra is $\Lie P={}^{\lambda}_{\ge 0}\Lg$. Similar to the Lie algebra case we put ${}^{\lambda}_{\ge0}G_{y,\ge r}:={}^{\lambda}_{\ge0}G\cap G_{y,\ge r}$. Note that ${}^{\lambda}_{\ge0}G_{y,0}={}^{\lambda}_{\ge0}G_{y,\ge0}/{}^{\lambda}_{\ge0}G_{y,>0}$ is the parabolic subgroup of $G_{y,0}$ associated to $\lambda$ (well-defined since $\lambda$ has image in $\bS\subset\bG_{y,0}$). We have $Z_G(e)\subset P$.\p

As explained in the construction of $\bG_{y,0}$ in Section \ref{secBT}. The torus $\bS$ over $k$ is canonically a maximal torus of $\bG_{y,0}$, and we shall realize $\lambda$ also as a cocharacter into $\bG_{y,0}$. For example, we have the corresponding parabolic subgroup ${}^{\lambda}_{\ge 0}\bG_{y,0}\subset\bG_{y,0}$ of $\bG_{y,0}$ and on the level of $k$-points ${}^{\lambda}_{\ge 0}G_{y,0}\subset G_{y,0}$. The Weyl group $W_y$ of $\bG_{y,0}$ and the Weyl group $W_P$ of the Levi of $P$ are subgroups of the Weyl group $W$ of $\bG$. Let $\{w_{\alpha}\}$ be a set of (lifts of) representatives for $W_y\bsl W/W_P$ in $N_G(S)$. We have the Iwasawa decomposition $G=\sqcup_{\alpha}G_{y,\ge0}w_{\alpha}P$. For any function $f\in C_c^{\infty}(\Lg)$, after averaging we assume $f$ is $G_{y,\ge0}$-conjugation invariant. The orbital integral $I_e(f)$ then can be expressed as a sum of integrals of $f$ on the $\Ad(w_{\alpha})P$-orbits of $\Ad(w_{\alpha})(e)$ for each $\alpha$. More precisely, denote by $\mu_e$ the measure on $\Ad(G)e$ that defines $I_e(-)$ and let $dg$ be the pullback of $\mu_e$ under $\Ad(g):G/Z_G(e)\xra{\sim}\Ad(G)e$. We may write
\[I_e(f)=\int_{G/Z_G(e)}f(\Ad(g)e)dg=\sum_{\alpha}\int_{G_{y,\ge 0}w_{\alpha}P/Z_G(e)}f(\Ad(g)e)dg.\]
Since $f$ is $G_{y,\ge0}$-conjugation invariant, for each index $\alpha$ there exists a left $P$-invariant measure $d_{\alpha}p$ (which plays the role of $dp^*$ in \cite[Lemma 3]{Ra72}) on $P/Z_G(e)$ such that
\begin{equation}\label{pushmea}
\int_{P/Z_G(e)}f(\Ad(w_{\alpha}p)e)d_{\alpha}p=\int_{G_{y,\ge 0}w_{\alpha}P/Z_G(e)}f(\Ad(g)e)dg.
\end{equation}
And $d_{\alpha}p$ is defined by this property for all $G_{y,\ge0}$-conjugation invariant $f$. The natural isomorphism $P/Z_G(e)\cong\Ad(P)(e)$ then allows us to push the measure $d_{\alpha}p$ to $\Ad(P)(e)$, which we denote by $\mu_{\alpha}$. We have
\begin{equation}\label{RR2}
I_e(f)=\sum_{\alpha}I_e(f)_{\alpha}\text{, where }I_e(f)_{\alpha}:=\int_{X\in\Ad(P)e}f(\Ad(w_{\alpha})X)\mu_{\alpha}.
\end{equation}

The main problem is to compute $\mu_{\alpha}$. We will deal with each $\alpha$ separately. Recall that $w_{\alpha}$ normalizes $S$ and note $\Ad(w_{\alpha})\Ad(P)e=\Ad(\Ad(w_{\alpha})P)(\Ad(w_{\alpha})e)$. By changing $e$ with $\Ad(w_{\alpha})e$ and $P$ with $\Ad(w_{\alpha})P$, without loss of generality we may and shall now assume $w_{\alpha}=1$. We note that $\Ad(P)(e)\subset{}^{\lambda}_{\ge 2}\Lg$ is open (in the $F$-analytic topology, \cite[4.14]{SS70}). For an open subgroup $K'\subset G$ such that $K'\cap P\supset G_{y,\ge0}\cap P$, by applying a function $f$ that takes $1$ on $\Ad(K')e$ and $0$ on $\Ad(G)(e)$ outside $\Ad(K')(e)$ to (\ref{pushmea}), we have
\[\mu_{\alpha}(\Ad(K'\cap P)e)=\mu_e(\Ad(K')e).\]
Since $\mu_{\alpha}$ (resp. $\mu_e$) is $P$-invariant (resp. $G$-invariant), for any open compact subgroup $K\subset G$ we deduce by comparing with $K'$ and $K'\cap K$ that
\begin{equation}\label{pushmea2}
\mu_{\alpha}(\Ad(K\cap P)e)=\frac{[G_{y,\ge0}:K]}{[G_{y,\ge0}\cap P:K\cap P]}\mu_e(\Ad(K)e).
\end{equation}
Here we use the notation $[\Lambda':\Lambda'']:=[\Lambda:\Lambda'']/[\Lambda:\Lambda']$ as relative index whenever $\Lambda',\Lambda''\subset\Lambda$ are finite index subgroups.\p

Let $\mu_{even}$ be the Haar measure on ${}^{\lambda}_{\ge2}\Lg$ that satisfies
\begin{equation}\label{even}
\mu_{even}({}^{\lambda}_{\ge2}\Lg_{y,>0})=\frac{|G_{y,0}|}{|{}^{\lambda}_{<0}\Lg_{y,0}|\cdot|{}^{\lambda}_{\ge0}G_{y,0}|}.
\end{equation}
Let $\mu_{odd}$ be the function on ${}^{\lambda}_{\ge2}\Lg$, which assigns to an element in $X+{}^{\lambda}_{>2}\Lg$, $X\in{}^{\lambda}_{2}\Lg$ the value
\[\left([\ad(X)({}^{\lambda}_{-1}\Lg_{y,\ge0}):{}^{\lambda}_{1}\Lg_{y,>0}]\right)^{1/2}.\]
The measure $\mu_{even}$ (resp. the function $\mu_{odd}$) is up to constant what is denoted $dXdZ$ (resp. the pullback of $\varphi(X)$ from ${}^{\lambda}_{2}\Lg$ to ${}^{\lambda}_{\ge2}\Lg$) in \cite[Theorem 1]{Ra72}.\p

\begin{lemma}\label{ted} Assume without loss of generality (see above) that $w_{\alpha}=1$. Then $\mu_{\alpha}$ is the restriction of $\mu_{odd}\cdot\mu_{even}$ to $\Ad(P)(e)$.
\end{lemma}

\begin{proof} By \cite[Thm. 1]{Ra72}, $\mu_{\alpha}$ differs from $\mu_{odd}\cdot\mu_{even}$ by a constant. Hence it suffices\footnote{In fact, one can also do the computation here in a more general setting and reproduce the proof of \cite[Thm. 1]{Ra72}.} to verify that they agree somewhere. Applying (\ref{pushmea2}) with $K=G_{y,>0}\subset G_{y,\ge 0}$ gives
\begin{equation}\label{L15}
\mu_{\alpha}(\Ad({}^{\lambda}_{\ge 0}G_{y,>0})e)=\frac{[G_{y,\ge0}:G_{y,>0}]}{[{}^{\lambda}_{\ge 0}G_{y,\ge0}:{}^{\lambda}_{\ge 0}G_{y,>0}]}\cdot\mu_e(\Ad(G_{y,>0})e)
\end{equation}
\[=[G_{y,\ge0}:G_{y,>0}\cdot{}^{\lambda}_{\ge 0}G_{y,\ge0}]\cdot\mu_e(\Ad(G_{y,>0})e)=[G_{y,0}:{}^{\lambda}_{\ge 0}G_{y,0}]\cdot\mu_e(\Ad(G_{y,>0})e).\]\hp

Note that $\mu_{odd}$ is constant on $\Ad({}^{\lambda}_{\ge 0}G_{y,>0})e$ since ${}^{\lambda}_{\ge 0}G_{y,>0}\subset G_{y,>0}$ is compact and $P$ (which acts through its Levi quotient $M\cong {}^{\lambda}_0G:=Z_G(\lambda)$) acts on $\mu_{odd}$ through a positive-valued quasi-character \cite[Lemma 2]{Ra72}. Thus to prove $\mu_{odd}\cdot\mu_{even}=\mu_{\alpha}$ on $\Ad(P)(e)$, it suffices to show
\begin{equation}\label{ted2}
\mu_{odd}(e)\mu_{even}(\Ad({}^{\lambda}_{\ge 0}G_{y,>0})e)=[G_{y,0}:{}^{\lambda}_{\ge 0}G_{y,0}]\cdot\mu_e(\Ad(G_{y,>0})e).
\end{equation}\hp

To emphasize the idea rather the heavy computation, let us first prove \ref{ted2} under simplifying assumptions. We assume
\renewcommand{\theenumi}{\roman{enumi}}
\begin{enumerate}
\item $\mu_{even}(\Ad({}^{\lambda}_{\ge 0}G_{y,>0})e)=\mu_{even}(\ad({}^{\lambda}_{\ge 0}\Lg_{y,>0})e)$.
\item $\mu_e(\Ad(G_{y,>0})e)=m_e(\mathrm{im}(\Lg_{y,>0}\ra\Lg/Z_{\Lg}(e)))$, where $m_e$ is the measure on $\Lg/Z_{\Lg}(e)$ given by (\ref{meam}), and the map is a restriction of the natural quotient map $\Lg\ra\Lg/Z_{\Lg}(e)$.
\item $\ad(e)({}^{\lambda}_{\ge 0}\Lg_{y,>0})={}^{\lambda}_{\ge 2}\Lg_{y,\ge0}$.
\item $\ad(e)({}^{\lambda}_{-1}\Lg_{y,>0})={}^{\lambda}_{1}\Lg_{y,\ge0}$.
\end{enumerate}\hp
Firstly, we rewrite (\ref{ted2}) using (i) and (ii) as
\begin{equation}\label{ted2*}
\mu_{odd}(e)\mu_{even}(\ad({}^{\lambda}_{\ge 0}\Lg_{y,>0})e)=[G_{y,0}:{}^{\lambda}_{\ge 0}G_{y,0}]\cdot m_e(\mathrm{im}(\Lg_{y,>0}\ra\Lg/Z_{\Lg}(e))).
\end{equation}
Next we prove (\ref{ted2*}) assuming (iii) and (iv). By (iii) we have
\[\mu_{odd}(e)=\left([\ad(e)({}^{\lambda}_{-1}\Lg_{y,\ge0}):{}^{\lambda}_{1}\Lg_{y,>0}]\right)^{1/2}=\left([\ad(e)({}^{\lambda}_{-1}\Lg_{y,>0}):{}^{\lambda}_{1}\Lg_{y,\ge0}]\right)^{1/2}\cdot|{}^{\lambda}_1\Lg_{y,0}|=|{}^{\lambda}_1\Lg_{y,0}|.\]
And by (iv)
\[
\mu_{even}(\ad({}^{\lambda}_{\ge 0}\Lg_{y,>0})e)=\frac{[\ad({}^{\lambda}_{\ge 0}\Lg_{y,>0})e:{}^{\lambda}_{\ge2}\Lg_{y,>0}] \cdot|G_{y,0}|}{|{}^{\lambda}_{<0}\Lg_{y,0}|\cdot|{}^{\lambda}_{\ge0}G_{y,0}|}
\]
\[
=\frac{[\ad(e){}^{\lambda}_{\ge 0}\Lg_{y,>0}:{}^{\lambda}_{\ge2}\Lg_{y,>0}] }{|{}^{\lambda}_{>0}\Lg_{y,0}|}\cdot\frac{|G_{y,0}|}{|{}^{\lambda}_{\ge0}G_{y,0}|}=\frac{[\ad(e){}^{\lambda}_{\ge 0}\Lg_{y,>0}:{}^{\lambda}_{\ge2}\Lg_{y,\ge0}] }{|{}^{\lambda}_1\Lg_{y,0}|}\cdot\frac{|G_{y,0}|}{|{}^{\lambda}_{\ge0}G_{y,0}|}=\frac{|G_{y,0}|}{|{}^{\lambda}_1\Lg_{y,0}|\cdot|{}^{\lambda}_{\ge0}G_{y,0}|}.
\]
This shows $\mu_{odd}(e)\mu_{even}(\Ad({}^{\lambda}_{\ge 0}G_{y,>0})e)=|G_{y,0}|/|{}^{\lambda}_{\ge0}G_{y,0}|=[G_{y,0}:{}^{\lambda}_{\ge0}G_{y,0}]$.
On the other hand, (iii) and (iv) of the above and (\ref{dual}) imply that $L:=\mathrm{im}(\Lg_{y,>0}\ra\Lg/Z_{\Lg}(e))$ is self-dual, i.e. $L^*=L$ where $L^*$ is as in (\ref{meam}). In particular $m_e(L)=1$. By (ii) this gives $\mu_e(\Ad(G_{y,>0})e)=1$. Thus \ref{ted2*} is proved under the assumptions.\p

Let us now prove (\ref{ted2*}) without assuming (iii) and (iv). The idea is that we had used (iii) and (iv) twice, and their effects actually cancel. More precisely, suppose instead of assuming (iii) and (iv), we have $C_1=[\ad(e){}^{\lambda}_{\ge 0}\Lg_{y,>0}:{}^{\lambda}_{\ge2}\Lg_{y,\ge0}]$ and $C_2=[\ad(e)({}^{\lambda}_{-1}\Lg_{y,>0}):{}^{\lambda}_{1}\Lg_{y,\ge0}]$. Then $\mu_{odd}(e)$ is multiplied by $C_2^{1/2}$ and $\mu_{even}(\Ad({}^{\lambda}_{\ge 0}G_{y,>0})e)$ is multiplied by $C_1$. On the other hand, since 
\[\ad(e)=\ad(e)|_{{}^\lambda_{\ge 0}\Lg}\oplus\ad(e)|_{{}^\lambda_{-1}\Lg}\oplus\ad(e)|_{{}^\lambda_{\le -2}\Lg}\]
and that $\ad(e)|_{{}^\lambda_{\ge 0}\Lg}:{}^\lambda_{\ge 0}\Lg\ra{}^\lambda_{\ge 2}\Lg$ is dual to $-\ad(e)|_{{}^\lambda_{\le -2}\Lg}:{}^\lambda_{\le -2}\Lg\ra{}^\lambda_{\le 0}\Lg$, we have
$[L^*:L]=C_1^{-2}C_2^{-1}$. Hence $m_e(L)=C_1C_2^{1/2}$, and thus both sides of (\ref{ted2*}) are (to be) multiplied by $C_1C_2^{1/2}$ and our argument worked without assumptions (iii) and (iv).\p

Recall that assumptions (i) and (ii) was used to rewrite the needed (\ref{ted2}) into (\ref{ted2*}). To get rid of the assumption, instead of having $K=G_{y,>0}$ in (\ref{L15}), we will apply (\ref{pushmea2}) with $K_r=G_{y,\ge r}$ with $r\gg 0$. Parallel to (\ref{ted2}) we have that Lemma \ref{ted} also follows from
\begin{equation}\label{ted3}
\mu_{odd}(e)\mu_{even}(\Ad({}^{\lambda}_{\ge 0}G_{y,\ge r})e)=[G_{y,\ge0}:G_{y,\ge r}\cdot{}^{\lambda}_{\ge 0}G_{y,\ge0}]\cdot\mu_e(\Ad(G_{y,\ge r})e).
\end{equation}
Instead of assumptions (i) and (ii) above, we have for $r\gg0$
\renewcommand{\theenumi}{\roman{enumi}$'$}
\begin{enumerate}
\item $\mu_{even}(\Ad({}^{\lambda}_{\ge 0}G_{y,\ge r})e)=\mu_{even}(\ad({}^{\lambda}_{\ge 0}\Lg_{y,\ge r})e)$.
\item $\mu_e(\Ad(G_{y,\ge r})e)=m_e(\mathrm{im}(\Lg_{y,\ge r}\ra\Lg/Z_{\Lg}(e)))$.
\end{enumerate}
They follow from (\ref{tangent}) and (\ref{tangent2}). Using (i$'$) and (ii$'$) we rewrite (\ref{ted3}) into
\begin{equation}\label{ted3*}
\mu_{odd}(e)\mu_{even}(\ad({}^{\lambda}_{\ge 0}\Lg_{y,\ge r})e)=[G_{y,\ge0}:G_{y,\ge r}\cdot{}^{\lambda}_{\ge 0}G_{y,\ge0}]\cdot m_e(\mathrm{im}(\Lg_{y,\ge r}\ra\Lg/Z_{\Lg}(e))).
\end{equation}
We claim that (\ref{ted3*}) is equivalent to (\ref{ted2*}) which was proved earlier, thus finishing the proof of Lemma \ref{ted}. Indeed, the difference of (\ref{ted2*}) and (\ref{ted3*}) is given by the ratio
\[
\frac{[\ad({}^{\lambda}_{\ge 0}\Lg_{y,>0})e:\ad({}^{\lambda}_{\ge 0}\Lg_{y,\ge r})e]}{[\mathrm{im}(\Lg_{y,>0}\ra\Lg/Z_{\Lg}(e)):\mathrm{im}(\Lg_{y,\ge r}\ra\Lg/Z_{\Lg}(e))]}
\cdot\frac{[G_{y,0}:{}^{\lambda}_{\ge0}G_{y,0}]}{[G_{y,\ge0}:G_{y,\ge r}\cdot{}^{\lambda}_{\ge 0}G_{y,\ge0}]}\]
\mhp
\[=\frac{[\ad(e){}^{\lambda}_{\ge 0}\Lg_{y,>0}:\ad(e){}^{\lambda}_{\ge 0}\Lg_{y,\ge r}]}{[\ad(e)\Lg_{y,>0}:\ad(e)\Lg_{y,\ge r}]}\cdot[G_{y,>0}:G_{y,\ge r}\cdot{}^{\lambda}_{\ge0}G_{y,>0}].
\]
Canceling the first fraction and applying (\ref{MP}) on the second part gives
\[=\frac{1}{[\ad(e){}^{\lambda}_{<0}\Lg_{y,>0}:\ad(e){}^{\lambda}_{<0}\Lg_{y,\ge r}]}\cdot[\Lg_{y,>0}:\Lg_{y,\ge r}\cdot{}^{\lambda}_{\ge0}\Lg_{y,>0}]=1
\]
where the last identity follows from the injectivity of $\ad(e)$ on ${}^{\lambda}_{<0}\Lg$.
\end{proof}\hp

In Section \ref{secdim} and \ref{secSha} we promised to prove Proposition \ref{comb}, (\ref{A1}), (\ref{A2}) and (\ref{A4}). We now give the proofs using (\ref{RR2}) and Lemma \ref{ted}.
\begin{proof}[Proof of Proposition \ref{comb}]
We apply (\ref{RR2}) with $y=x$ and $f=1_{\Lg_{x,\ge 0}}$. We have $\Lg_{x,0}={}^{\lambda}_{0}\Lg_{x,0}$, and ${}^{\lambda}_{<0}\Lg_{x,0}=0$. This implies $\ad(X)({}^{\lambda}_{-1}\Lg_{x,\ge0})=\ad(X)({}^{\lambda}_{-1}\Lg_{x,>0})\subset{}^{\lambda}_{1}\Lg_{x,>0}$. Hence $\mu_{odd}\le 1$ on $\Lg_{x,\ge 0}$, and consequently $I_e(1_{\Lg_{x,\ge0}})_{\alpha}\le\mu_{even}(\Lg_{x,>0})=1$ for all $\alpha\in W_x\bsl W/W_P=W/W_P$. This proves the proposition.
\end{proof}\hp

\begin{proof}[Proof of (\ref{A1}) and (\ref{A2})]
Equation (\ref{A2}) is a direct result of Proposition \ref{comb} and Lemma \ref{dilate}.
For (\ref{A1}), by Lemma \ref{dilate} it suffices to prove $I_{te_1}(1_{\Lg_{x,\ge 0}})=I_{e_1}(1_{\Lg_{x,\ge 0}})=|W|$. We apply (\ref{RR2}) with $y=x$ and $e=e_1$. Note that when $e$ is regular, $\mu_{odd}$ is trivial as $\lambda$ has no odd weight space\footnote{This is because $\lambda$ has weight $2$ on all simple roots; $e$ can taken to be a sum of non-trivial elements from each simple root space. Now all roots are generated by simple roots and thus have even weights.}, and $\mu_{even}({}^{\lambda}_{>0}\Lg_{y,\ge0})=\mu_{even}({}^{\lambda}_{\ge2}\Lg_{y,>0})=1$ since $y=x$ lies in an alcove. This gives $I_{e_1}(1_{\Lg_{x,\ge 0}})=\sum_{\alpha\in W_y\bsl W/W_p}1$. Now $W_y=W_x=1$ as $x$ lies in an alcove and $W_P=1$ as $P$ is a Borel, hence the result.
\end{proof}\hp


\begin{proof}[Proof of (\ref{A4})]
The proof of (\ref{A4}) will be the longest. It is in spirit close of how the Steinberg variety is used to bound the dimension of Springer fibers (e.g. the proof of \cite[Cor. 3.3.24]{CG10}). We apply (\ref{RR2}) with $y=x_i$, $e=e_j$, $f=f_i$ and for convenience write $\sigma=\bar{e}_i$. In particular $\sigma$ is a nilpotent element in $\Lg_{y,0}$. Let us also write $d:=\dim\Ad(\bG_{y,0})\sigma$ ($=d_i$ in the old setting at (\ref{A4})). Let $f^{\sharp}$ be the function that takes value $q^{-d/2}$ on $\Ad(G_{y,0})\sigma+\Lg_{y,>0}$. This function is somewhat in between $f_i'$ and $f_i^*$ in Lemma \ref{L12o} and \ref{L12a}, and with the same proof as theirs we have
\begin{lemma}\label{L19}
 $I_e(f)/I_e(f^{\sharp})=O(1)$.
\end{lemma}
It then suffices to prove $I_e(f^{\sharp})=O(1)$. By (\ref{RR2}) it suffices to prove each $I_{e_j}(f^{\sharp})_{\alpha}=O(1)$ using Lemma \ref{ted}, where we likewise assume $w_{\alpha}=1$ without loss of generality. Note that $\mu_{even}({}^{\lambda}_{\ge2}\Lg_{y,>0})=1+O(q^{-1/2})$ by the very definition in (\ref{even}) (thanks to the basic (\ref{basicest})). We thus arrive at
\begin{lemma} To prove (\ref{A4}), it suffices to prove
\begin{equation}\label{L20}
I_{e_j}(f^{\sharp})_{\alpha}/\mu_{even}({}^{\lambda}_{\ge2}\Lg_{y,>0})=O(1).
\end{equation}
\end{lemma}
Recall that $\lambda:\mathbb{G}_m\ra\bS$ is associated to $e$ by the Jacobson-Morozov  theorem. For $X\in{}^{\lambda}_{2}\Lg_{y,\ge0}$, we have $\ad(X)({}^{\lambda}_{-1}\Lg_{y,>0})\subset{}^{\lambda}_1\Lg_{y,>0}$ and thus
\[\mu_{odd}(X)\le[\ad(X)({}^{\lambda}_{-1}\Lg_{y,\ge0}):\ad(X)({}^{\lambda}_{-1}\Lg_{y,>0})]^{1/2}= q^{\frac{1}{2}\dim\ad(X)({}^{\lambda}_{-1}\Lg_{y,0})}.\]
Denote by $\mu_{odd}^*(\bar{e}):=q^{\frac{1}{2}\dim\ad({}^{\lambda}_{2}\bar{e})({}^{\lambda}_{-1}\Lg_{y,0})}$ for any $\bar{e}\in{}^{\lambda}_{\ge2}\Lg_{y,0}$, where ${}^{\lambda}_{2}\bar{e}$ is the projection of $\bar{e}$ in ${}^{\lambda}_{2}\Lg_{y,0}$. Now we begin our proof of (\ref{L20}) as
\[
\frac{I_{e_j}(f^*)_{\alpha}}{\mu_{even}({}^{\lambda}_{\ge2}\Lg_{y,>0})}\le q^{-\frac{1}{2}\dim\Ad(\bG_{y,0})\sigma}\cdot\sum_{\bar{e}\in\Ad(G_{y,0})\sigma\cap{}^{\lambda}_{\ge2}\Lg_{y,0}}\mu_{odd}^*(\bar{e})\mu_{even}({}^{\lambda}_{\ge2}\Lg_{y,>0}).\]
\[=q^{-\frac{1}{2}\dim\Ad(\bG_{y,0})\sigma}\cdot\frac{|{}^{\lambda}_{\ge0}G_{y,0}|}{|G_{y,0}|}\cdot\sum_{g\in G_{y,0}/{}^{\lambda}_{\ge0}G_{y,0}}\left(\sum_{\bar{e}\in\Ad(G_{y,0})\sigma\cap\Ad(g)({}^{\lambda}_{\ge2}\Lg_{y,0})}\mu_{odd}^*(\Ad(g^{-1})\bar{e})\right)
\]
\[=q^{-\frac{1}{2}\dim\Ad(\bG_{y,0})\sigma}\cdot\frac{|{}^{\lambda}_{\ge0}G_{y,0}|}{|Z_{G_{y,0}}(\sigma)|}\cdot\sum_{g\in G_{y,0}/{}^{\lambda}_{\ge0}G_{y,0},\;\Ad(g)({}^{\lambda}_{\ge2}\Lg_{y,0})\ni\sigma}\mu_{odd}^*(\Ad(g^{-1})\sigma)
\]
\[=q^{-\frac{1}{2}\dim\Ad(\bG_{y,0})\sigma}\cdot\frac{|{}^{\lambda}_{\ge0}G_{y,0}|}{|Z_{G_{y,0}}(\sigma)|}\cdot\]
\[\sqrt{\sum_{g_1,g_2\in G_{y,0}/{}^{\lambda}_{\ge0}G_{y,0},\;\Ad(g_1)({}^{\lambda}_{\ge2}\Lg_{y,0})\cap\Ad(g_2)({}^{\lambda}_{\ge2}\Lg_{y,0})\ni\sigma}\mu_{odd}^*(\Ad(g_1^{-1})\sigma)\mu_{odd}^*(\Ad(g_2^{-1})\sigma)}
\]
\[= q^{-\frac{1}{2}\dim\Ad(\bG_{y,0})\sigma}\cdot\frac{|{}^{\lambda}_{\ge0}G_{y,0}|}{|Z_{G_{y,0}}(\sigma)|\cdot\sqrt{|\Ad(G_{y,0})\sigma|}}\cdot\]
\[\sqrt{\sum_{g_1,g_2\in G_{y,0}/{}^{\lambda}_{\ge0}G_{y,0},\;\bar{e}\in\Ad(G_{y,0})\sigma,\;\Ad(g_1)({}^{\lambda}_{\ge2}\Lg_{y,0})\cap\Ad(g_2)({}^{\lambda}_{\ge2}\Lg_{y,0})\ni\bar{e}}\mu_{odd}^*(\Ad(g_1^{-1})\sigma)\mu_{odd}^*(\Ad(g_2^{-1})\sigma)}
\]
\[\le q^{-\frac{1}{2}\dim\Ad(\bG_{y,0})\sigma}\cdot\frac{|{}^{\lambda}_{\ge0}G_{y,0}|}{|Z_{G_{y,0}}(\sigma)|\cdot\sqrt{|\Ad(G_{y,0})\sigma|}}\cdot\]
\[\sqrt{\sum_{g_1,g_2\in G_{y,0}/{}^{\lambda}_{\ge0}G_{y,0}}\left(\sum_{\bar{e}\in\Ad(g_1)({}^{\lambda}_{\ge2}\Lg_{y,0})\cap\Ad(g_2)({}^{\lambda}_{\ge2}\Lg_{y,0})}\mu_{odd}^*(\Ad(g_1^{-1})\bar{e})\mu_{odd}^*(\Ad(g_2^{-1})\bar{e})\right)}.
\]\p

Here the last inequality is given by removing the constraint $\bar{e}\in\Ad(\bG_{y,0})\sigma$. We want to prove that the above huge term is of the order $O(1)$. We can simplify the sum over $g_1,g_2$ using the Bruhat decomposition: the orbits of $(g_1,g_2)$ (under left $G_{y,0}$-action) are represented by $(1,w)$ with $w\in{}^{\lambda}_{\ge0}G_{y,0}\bsl G_{y,0}/{}^{\lambda}_{\ge0}G_{y,0}\cong W_P\bsl W_y/W_P$. We choose $w$ to be a representative that normalizes $\bS\subset\bG_{y,0}$, and write $w\lambda:=\Ad(w)\lambda$ again a cocharacter. Now $\lambda$ and $w\lambda$ both have image in $\bS$ and give gradings on $\Lg_{y,0}$ so that we have 
\[\Lg_{y,0}=\bigoplus_i\bigoplus_j{}^{w\lambda}_i\left({}^{\lambda}_j\Lg_{y,0}\right)\text{, where }{}^{w\lambda}_i\left({}^{\lambda}_j\Lg_{y,0}\right):={}^{w\lambda}_i\Lg_{y,0}\cap{}^{\lambda}_j\Lg_{y,0}.\]
There is a $\bG_{y,0}$-invariant form on $\Lg_{y,0}$ that necessarily have weight $0$ in these grading. Thus we have
\begin{equation}\label{dualw}
{}^{\lambda}_j\Lg_{y,0}\cong({}^{\lambda}_{-j}\Lg_{y,0})^*\text{ and }{}^{w\lambda}_i\left({}^{\lambda}_j\Lg_{y,0}\right)\cong\left({}^{w\lambda}_{-i}\left({}^{\lambda}_{-j}\Lg_{y,0}\right)\right)^*\text{, etc.}
\end{equation}
The orbit of $(1,w)$ has size $[{}^{\lambda}_{\ge0}G_{y,0}\cdot w\cdot{}^{\lambda}_{\ge0}G_{y,0}:{}^{\lambda}_{\ge0}G_{y,0}]=[{}^{\lambda}_{\ge0}G_{y,0}:{}^{\lambda}_{\ge0}G_{y,0}\cap{}^{w\lambda}_{\ge0}G_{y,0}]=O(1)\cdot[{}^{\lambda}_{>0}\Lg_{y,0}:{}^{w\lambda}_{>0}({}^{\lambda}_{>0}\Lg_{y,0})]$. On the other hand, we have
\[q^{-\frac{1}{2}\dim\Ad(\bG_{y,0})\sigma}\cdot\frac{|{}^{\lambda}_{\ge0}G_{y,0}|}{\sqrt{|G_{y,0}|\cdot|Z_{G_{y,0}}(\sigma)|}}=O(1)\cdot(q^{-\dim{}^{\lambda}_{<0}\Lg_{y,0}})=O(1)\cdot(|{}^{\lambda}_{>0}\Lg_{y,0}|^{-1}),\]
where we use (\ref{dualw}) in the last equality. We are thus reduced to prove
\[\sum_{\bar{e}\in{}^{w\lambda}_{\ge2}({}^{\lambda}_{\ge2}\Lg_{y,0})}\mu_{odd}^*(\bar{e})\mu_{odd}^*(\Ad(w^{-1})\bar{e})\le|{}^{w\lambda}_{>0}({}^{\lambda}_{>0}\Lg_{y,0})|.\]\p

When $\lambda$ has no odd weight spaces, ${}^{\lambda}_{>0}\Lg_{y,0}={}^{\lambda}_{\ge2}\Lg_{y,0}$. Conjugation by $w$ implies also ${}^{w\lambda}_{>0}\Lg_{y,0}={}^{w\lambda}_{\ge2}\Lg_{y,0}$. Hence ${}^{w\lambda}_{>0}({}^{\lambda}_{>0}\Lg_{y,0})={}^{w\lambda}_{\ge2}({}^{\lambda}_{\ge2}\Lg_{y,0})$. Moreover $\mu_{odd}^*\equiv 1$, hence the above inequality holds as an equality in this case. In general when ${}^{\lambda}_{1}\Lg_{y,0}$ is possibly non-trivial, it suffices to prove that for all $\bar{e}\in{}^{w\lambda}_{\ge2}({}^{\lambda}_{\ge2}\Lg_{y,0})$ we have \[\mu_{odd}^*(\bar{e})\mu_{odd}^*(\Ad(w^{-1})\bar{e})\le\frac{|{}^{w\lambda}_{>0}({}^{\lambda}_{>0}\Lg_{y,0})|}{|{}^{w\lambda}_{\ge2}({}^{\lambda}_{\ge2}\Lg_{y,0})|}=|{}^{w\lambda}_{1}({}^{\lambda}_{\ge2}\Lg_{y,0})|\cdot|{}^{w\lambda}_{\ge2}({}^{\lambda}_{1}\Lg_{y,0})|\cdot|{}^{w\lambda}_{1}({}^{\lambda}_{1}\Lg_{y,0})|.\]
Squaring each side, the above inequality will follow from a combination of
\[\left(\mu_{odd}^*(\Ad(w^{-1})\bar{e})\right)^2\le|{}^{w\lambda}_1({}^{\lambda}_{\ge2}\Lg_{y,0})|^2\cdot|{}^{w\lambda}_{1}({}^{\lambda}_{1}\Lg_{y,0})|\]
and
\[\left(\mu_{odd}^*(\bar{e})\right)^2\le|{}^{w\lambda}_{\ge2}({}^{\lambda}_{1}\Lg_{y,0})|^2\cdot|{}^{w\lambda}_{1}({}^{\lambda}_{1}\Lg_{y,0})|.\]
We note that the second inequality becomes the first by replacing $\bar{e}=\Ad(w^{-1})\bar{e}$, $\lambda=w^{-1}\lambda$ and $w=w^{-1}$. We will prove the second inequality, which from the definition of $\mu_{odd}^*$, is the same as
\begin{equation}\label{last}\dim\ad({}^{\lambda}_2\bar{e})({}^{\lambda}_{-1}\Lg_{y,0})\le 2\dim{}^{w\lambda}_{\ge2}({}^{\lambda}_{1}\Lg_{y,0})+\dim{}^{w\lambda}_{1}({}^{\lambda}_{1}\Lg_{y,0}).
\end{equation}\p

Write $V=\ad({}^{\lambda}_2\bar{e})({}^{\lambda}_{-1}\Lg_{y,0})\subset{}^{\lambda}_1\Lg_{y,0}$, and consider $\pi:{}^{\lambda}_1\Lg_{y,0}\ra{}^{w\lambda}_{\le0}\left({}^{\lambda}_1\Lg_{y,0}\right)$ the natural projection. Since $\bar{e}\in{}^{w\lambda}_{\ge 2}\Lg_{y,0}$, we have $\pi(\ad({}^{\lambda}_2\bar{e})({}^{w\lambda}_{\ge-1}({}^{\lambda}_{-1}\Lg_{y,0})))=0$. Hence $\dim(\pi(V))\le\dim{}^{w\lambda}_{\le-2}({}^{\lambda}_{-1}\Lg_{y,0})=\dim{}^{w\lambda}_{\ge2}({}^{\lambda}_{1}\Lg_{y,0})$ by (\ref{dualw}). We then have
\[\begin{array}{ccc}
\dim V=\dim(\pi(V))+\dim(\ker(\pi|_V))&\le&\dim{}^{w\lambda}_{\ge2}({}^{\lambda}_{1}\Lg_{y,0})+\dim{}^{w\lambda}_{\ge1}({}^{\lambda}_{1}\Lg_{y,0})
\\&=&2\dim{}^{w\lambda}_{\ge2}({}^{\lambda}_{1}\Lg_{y,0})+\dim{}^{w\lambda}_{1}({}^{\lambda}_{1}\Lg_{y,0}).
\end{array}\]
which is exactly (\ref{last}). This finishes the proof of (\ref{A4}).
\end{proof}\p

\section{A similarity with components of the Steinberg variety}\label{secSt}

This section is independent from the main result (Theorem \ref{main}). Here we discuss extra results along the method which hint that the $|W|$ orbits of components in the main result might be related to the fact that the Steinberg variety also have $|W|$ components (which has an explicit construction, see e.g. \cite[Cor. 3.3.5]{CG10}). We also propose two conjectures partially inspired by the results. Let $x\in\CA_{\bS}$ be in an alcove as before and $y\in\CA_{\bS}$ be contained in the closure of the alcove. Fix $\gamma=t\gamma_0$ where $\gamma_0\in\Lg$ is topologically nilpotent and regular semisimple as in Section \ref{secSha}.\p

To begin with, we need a generalization of (\ref{GKM3}). Let $\bar{e}\in\Lg_{y,0}$ be a nilpotent element. Recall that in (\ref{ASF2}) we have the affine Springer fiber $\CX_{y,\gamma}$ with
\[\CX_{y,\gamma}(k)=\{g\in G/G_{y,\ge 0}\;|\;\Ad(g^{-1})\gamma\in\Lg_{y,\ge0}\},\]
and similarly for any $k'/k$ finite by taking the corresponding unramified base change of $F$. Denote by $\theta:\Lg_{y,\ge 0}\sra\Lg_{y,0}$ the natural projection. Then the map $(g\mapsto\theta(\Ad(g^{-1})\gamma))$ gives a natural map $\phi$ from $\Lg_{y,\gamma}$ to the stack $[\Lie\bG_{y,0}/\bG_{y,0}]$, where $\bG_{y,0}$ acts on the affine variety $\Lie\bG_{y,0}$ by $\Ad$. Let $\bar{e}\in\Lg_{y,0}$ be any nilpotent element. Then $[\Ad(\bG_{y,0})\bar{e}/\bG_{y,0}]\subset[\Lie\bG_{y,0}/\bG_{y,0}]$ is a locally closed substack, and the generalized affine Springer fiber $\CX_{y,\bar{e},\gamma}$ can be identified as the preimage of this substack under $\phi$. Note that as $\Ad(\bG_{y,0})\bar{e}=Z_{\bG_{y,0}}(\bar{e})\bsl G$, we have $[\Ad(\bG_{y,0})\bar{e}/\bG_{y,0}]\cong [\Spec k/Z_{\bG_{y,0}}(\bar{e})]$.\p

Write $A(\bar{e}):=\pi_0(Z_{\bG_{y,0}}(\bar{e}))$. Identify $\bar{e}\in[\Ad(\bG_{y,0})\bar{e}/\bG_{y,0}](k)$ and we have a natural map
\[\pi_1([\Ad(\bG_{y,0})\bar{e}/\bG_{y,0}],\bar{e})\xra{\sim}\pi_1([\Spec k/Z_{\bG_{y,0}}(\bar{e})],\Spec k).\]
We now take a finite base change of $k$ so that $\Gal(\bar{k}/k)$ acts trivially on $A(\bar{e})$. In this case there is a natural map $\pi_1([\Spec k/Z_{\bG_{y,0}}(\bar{e})],\Spec k)\sra A(\bar{e})=\pi_0(Z_{\bG_{y,0}}(\bar{e}))$. For any ($\ell$-adic) $\eta\in \Irr(A(\bar{e}))$, the above map induces a local system $\CL_{\eta}^o$ on $[\Ad(\bG_{y,0})\bar{e}/\bG_{y,0}]$. This local system has the following property: for any $\bar{e}'\in(\Ad(\bG_{y,0})\bar{e})(k)$, there exists $\bar{g}\in\bG_{y,0}(\bar{k})$ such that $\Ad(\bar{g}^{-1})\bar{e}=\bar{e}'$. Then since both $\bar{e}$ and $\bar{e}'$ are defined over $k$, we have $\Ad((\sigma.\bar{g})^{-1}\bar{g})\bar{e}=\bar{e}$ for any $\sigma\in\Gal(\bar{k}/k)$. This gives a cohomology class $(\sigma\mapsto(\sigma.\bar{g})^{-1}\bar{g})\in H^1(k,Z_{\bG_{y,0}}(\bar{e}))\cong H^1(k,A(\bar{e}))$, where the last isomorphism follows from Lang's theorem. Note that different choices of $\bar{g}$ gives the same cohomology class. This element in $H^1(k,A(\bar{e}))$ can be identified with a conjugacy class $\tau_{\bar{e}'}$ in $A(\bar{e})$. We then have
\begin{equation}\label{sheaf0}
\Tr(\Frob:(\CL_{\eta}^o)_{\bar{e}'})=\Tr(\eta(\tau_{\bar{e}'})).
\end{equation}
Denote by $\CL_{\eta}$ the pullback of $\CL_{\eta}^o$ to $\CX_{y,\bar{e},\gamma}$ under $\phi$. Note that since $\phi:\CX_{g,\bar{e},\gamma}\ra[\Ad(\bG_{y,0})\bar{e}/\bG_{y,0}]$ is $\CT$-invariant, $\CL_{\eta}$ is naturally a $\CT$-equivariant local system. 
For any $g\in\CX_{y,\bar{e},\gamma}(k)$ we likewise have
\begin{equation}\label{sheaf}
\Tr(\Frob:(\CL_{\eta})_g)=\Tr(\eta(\tau_{\theta(\Ad(g^{-1})\gamma)})).
\end{equation}
\p

Let $\bar{f}_{\eta}$ be the function on $\Lg_{y,0}$ which takes the value $\bar{f}_{\eta}(\bar{e}')=\Tr(\eta(\tau_{\bar{e}'}))$ for $\bar{e}'\in(\Ad(\bG_{y,0})\bar{e})(k)$ and zero otherwise. Let $f_{\eta}$ be the inflation of $\bar{f}_{\eta}$ to $\Lg_{y,\ge 0}$. We have

\begin{lemma}\label{GKM4l}  We have a generalization of (\ref{GKM3}) as
\begin{equation}\label{GKM4}
I_{\gamma}^{st}(f_{\eta})=\frac{|G_{y,0}|}{|T_0|}\cdot q^{(-v(D(\gamma))-\dim\bG_{y,0}+\dim\bT_0)/2}\cdot\mathrm{Tr}(\mathrm{Frob}\,;\,H_c^*(\CX_{y,\bar{e},\gamma}\times\Spec\bar{k},\CL_{\eta})^{\CT}).
\end{equation}
\end{lemma}
\begin{proof} When $\eta=1$, this is just (\ref{GKM3}). In Section \ref{secgeom} we explained that (\ref{GKM3}) is proved with the same proof as that of \cite[Theorem 3.4.8]{Yu16} with the affine Springer fiber replaced by $\CX_{y,\bar{e},\gamma}$. For (\ref{GKM4}) we furthermore replace the constant sheaf $\underline{\mathbb{Q}_{\ell}}$ by $\CL_{\eta}$. Applying Grothendieck-Lefschetz trace formula to the sheaf $\CL_{\eta}$ on the stack\footnote{Here $\til{\Lambda}$ is an \'{e}tale sub-group scheme of $\CT$ that maps onto $\pi_0(\CT)$, see {\it loc. cit.}} $[\tilde\Lambda\bsl\CX_{y,\bar{e},\gamma}]$ as after \cite[(3.4.15)]{Yu16}, one has that the main term  $\mathrm{Tr}(\mathrm{Frob}\,;\,H_c^*(\CX_{y,\bar{e},\gamma}\times\Spec\bar{k},\CL_{\eta})^{\CT})$ of the RHS of (\ref{GKM4}) is equal to a sum of $\Tr(\Frob:(\CL_{\eta})_g)$ over $g\in[\tilde\Lambda\bsl\CX_{y,\bar{e},\gamma}](k)$. By (\ref{sheaf}) this is the sum of $\Tr(\eta(\tau_{\bar{e}_g}))$ over the same set. The same discussion in {\it loc. cit.} shows that this is equal to the LHS of (\ref{GKM4}) up to the same normalization factor $\frac{|G_{y,0}|}{|T_0|}\cdot q^{(-v(D(\gamma))-\dim\bG_{y,0}+\dim\bT_0)/2}$ that also appeared in (\ref{GKM3}).\end{proof}
\p

We denote by $\rho_{\eta}$ the representation of $W_y$ (=the Weyl group of $\bG_{y,0}$) associated to $(\Ad(\bG_{y,0})\bar{e},\eta)$ by the original Springer correspondence (see e.g. \cite[1.5]{Yu16} or \cite[3.6]{CG10}). Here we define $\rho_{\eta}:=0$ if $(\Ad(\bG_{y,0})\bar{e},\eta)$ is not in the image of the Springer correspondence. We can now state the following generalization of Theorem \ref{main} with the local system $\CL_{\eta}$. 

\begin{theorem}\label{Steinberg}$\,$ Write $d_{\bar{e}}:=\frac{1}{2}(\dim Z_{\bG_{y,0}}(\bar{e})-\rank\bG_{y,0})$ the dimension of the Springer fiber above $\bar{e}\in\Lg_{y,0}$, and $d:=\dim\CX_{\gamma}-d_{\bar{e}}$. We have
\[\dim\CX_{y,\bar{e},\gamma}=d\text{, and }\dim H_c^{2d}(\CX_{y,\bar{e},\gamma}\times\Spec\bar{k};\CL_{\eta})^{\CT}=\frac{|W|}{|W_y|}\dim\rho_{\eta}.\]
\end{theorem}

\begin{proof} For the first identity, the $\le$-part is obvious as we have in (\ref{mapSpr}) a natural map $\CX_{\gamma}\ra\CX_{y,\gamma}$ whose fiber over any point in $\CX_{y,\bar{e},\gamma}\subset\CX_{y,\gamma}$ is isomorphic to a Springer fiber above $\bar{e}$ and have dimension $d_{\bar{e}}$. Consequently the LHS of the second identity is (if non-zero) the top degree cohomology, and the second identity then implies the first as $\dim\rho_{\eta}>0$ at least when $\eta$ is trivial. By taking a further base change of $k$, we may assume that $\Gal(\bar{k}/k)$ acts trivially on the set of components of the Springer fiber above $\bar{e}$. Let $f_{\eta}$ be as in Lemma \ref{GKM4l} and $f_{\eta}^*:=q^{-\frac{1}{2}\dim\Ad(\bG_{y,0})\bar{e}}f_{\eta}$. We would like to prove
\begin{equation}
\label{St0}I_{\gamma}^{st}(f_{\eta}^*)=\frac{|W|}{|W_y|}\dim\rho_{\eta}+O(q^{-1/2}).
\end{equation}

\begin{lemma} The second identity in Theorem \ref{Steinberg} follows from (\ref{St0}).
\end{lemma}

\begin{proof} Plugging (\ref{St0}) into (\ref{GKM4}) we have
\begin{equation}\label{Stmid}
\mathrm{Tr}(\mathrm{Frob}\,;\,H_c^*(\CX_{y,\bar{e},\gamma}\times\Spec\bar{k},\CL_{\eta})^{\CT})=q^D\cdot\frac{|W|}{|W_y|}\cdot(1+O(q^{-1/2}))
\end{equation}
where
\[D:=\frac{1}{2}(v(D(\gamma))-\dim\bG_{y,0}+\dim\bT_0+\dim\Ad(\bG_{y,0})\bar{e}).\]
Using (\ref{Bezr}), the last equation is equal to
\[=\frac{1}{2}(2\dim\CX_{\gamma}+\rank\bG_{y,0}-\dim\bG_{y,0}+\dim\Ad(\bG_{y,0})\bar{e})\]
and by the first identity in Theorem \ref{Steinberg} that $\dim\CX_{\gamma}=\dim\CX_{y,\bar{e},\gamma}+2d_{\bar{e}}$, it is equal to
\[=\frac{1}{2}(2\dim\CX_{y,\bar{e},\gamma}+\dim Z_{\bG_{y,0}}(\bar{e})-\dim\bG_{y,0}+\dim\Ad(\bG_{y,0})\bar{e})=\dim\CX_{y,\bar{e},\gamma}.\]
In the LHS of (\ref{Stmid}), $\Frob$ acts on the top cohomology $H^{2d}_c(-)$ with weight $2d$ and $H^{<2d}_c(-)$ with smaller weights. Since (\ref{Stmid}) holds with $D=d$ for any finite base change of $k$ (i.e. the same estimate holds for $\Frob^n$, $n=1,2,...$), the result follows.
\end{proof}\hp

The proof of equation (\ref{St0}) will be a modification of the proof of (\ref{maineq0}). The modification needed is to change the function from $1_{\Lg_{x,\ge 0}}$ to $f_{\eta}^*$, $1_{\Lg_{x,\ge -1}}$ by $f_{\eta}^*$ dilated by $t^{-1}$, and $|W|$ (in the RHS of (\ref{maineq0}) and (\ref{A1})) to $\frac{|W|}{|W_y|}\dim\rho_{\eta}$. Equation (\ref{21b}) is then changed accordingly. Among (\ref{A1})$\sim$(\ref{A6}), the last four equations remain unchanged. For the first two, recall that thanks to Lemma \ref{dilate}, (\ref{A1}) and (\ref{A2}) are equivalent after dilating by $t$ to
\begin{align}
I_{e_1}(1_{\Lg_{x,\ge 0}})&=|W|.\\
I_{e_i}(1_{\Lg_{x,\ge 0}})&=O(1).
\end{align}
Here recall $e_1\in\Lg$ is a regular nilpotent element and $e_i\in\Lg$ is an arbitrary nilpotent element. For (\ref{St0}) they are to be replaced by
\begin{align}
\label{A7}I_{e_1}(f_{\eta}^*)&=\frac{|W|}{|W_y|}\dim\rho_{\eta}+O(q^{-1/2}).\\
\label{A8}I_{e_i}(f_{\eta}^*)&=O(1).
\end{align}

It remains to prove (\ref{A7}) and (\ref{A8}). For (\ref{A8}), we note that $f_{\eta}^*$ is very similar to the function $f_i^*$ in Section \ref{secSha}. Let $y=x_i$ and $\bar{e}=\bar{e}_i$. Then $f_{\eta}^*=f_i^*$ if $\eta$ is trivial. In general, since the trace value of $\eta$ is bounded (by $\dim\eta$), we have $|I_{e_i}(f_{\eta}^*)|\le\dim\eta\cdot I_{e_i}(f_i^*)=O(1)$ thanks to (\ref{A4}).\p

To prove (\ref{A7}), as in the proof of (\ref{A1}) in Section \ref{secRR} we rely on (\ref{RR2}) and Lemma \ref{ted}. Here we use the same $y$ here for the $y$ in the setting of (\ref{RR2}). We have $G_{y,\ge0}\bsl G/P\cong W_y\bsl W$ where $P\subset G$ is the parabolic associated to $e_1$, i.e. a Borel subgroup. For any of the double cosets, the measure $\mu_{even}({}^{\lambda}_{\ge 2}\Lg_{y,>0})$ in Lemma \ref{ted} is $1+O(q^{-1/2})$. Here $\lambda$ is the cocharacter associated to the regular nilpotent $e_1$. In particular ${}^{\lambda}_{\ge 2}\Lg={}^{\lambda}_{>0}\Lg$. For every (conjugacy class of) $\alpha\in A(\bar{e})$, let $\bar{e}_{\alpha}$ be an arbitrary choice of an element in $(\Ad(\bG_{y,0})\bar{e})(k)$ of class $\alpha$, i.e. such that $\tau_{\bar{e}_{\alpha}}=\alpha$. For any double coset in $G_{y,\ge0}\bsl G/P$, the contribution (see (\ref{RR2})) to $I_{e_1}(f_{\eta}^*)$ is, up to the $(1+O(q^{-1/2}))$-factor $\mu_{even}({}^{\lambda}_{\ge 2}\Lg_{y,>0})$, given by
\begin{equation}\label{regterm}
\sum_{\epsilon\in(\Ad(\bG_{y,0})\bar{e})(k)\cap{}^{\lambda}_{>0}\Lg_{y,0}}f_{\eta}^*(\epsilon)=q^{-\frac{1}{2}\dim\Ad(\bG_{y,0})\bar{e}}\cdot\sum_{\alpha\in A(\bar{e})/conj}\Tr(\eta(\alpha))\cdot|\Ad(G_{y,0})\bar{e}_{\alpha}\cap{}^{\lambda}_{>0}\Lg_{y,0}|.
\end{equation}
We have
\[|\Ad(G_{y,0})\bar{e}_{\alpha}\cap{}^{\lambda}_{>0}\Lg_{y,0}|\cdot\frac{|G_{y,0}|}{|{}^{\lambda}_{\ge0}G_{y,0}|}=|\CB_{\bar{e}_{\alpha}}(k)|\cdot\frac{|G_{y,0}|}{|Z_{G_{y,0}}(\bar{e}_{\alpha})|}\]
since both sides counts the number of $(\bar{e}',\mathfrak{b})$ where $\bar{e}'\in\Ad(G_{y,0})\bar{e}_{\alpha}$ and $\mathfrak{b}\subset\Lg_{y,0}$ is a $k$-Borel subalgebra so that $\bar{e}'\in\mathfrak{b}$. Rearranging the last identity gives
\begin{equation}\label{e0}
|\Ad(G_{y,0})\bar{e}_{\alpha}\cap{}^{\lambda}_{>0}\Lg_{y,0}|=\frac{|{}^{\lambda}_{\ge0}G_{y,0}|}{|Z_{G_{y,0}}(\bar{e}_{\alpha})|}\cdot|\CB_{\bar{e}_{\alpha}}(k)|
\end{equation}
We proceed to estimate the three terms on the RHS. Firstly, by (\ref{basicest}) we have:
\begin{equation}\label{e1}
|{}^{\lambda}_{\ge0}G_{y,0}|=q^{\dim{}^{\lambda}_{\ge0}\bG_{y,0}}\cdot(1+O(q^{-1/2})).
\end{equation}
Next we look at the term $|Z_{G_{y,0}}(\bar{e}_{\alpha})|$. The abstract group of geometric components of $Z_{\bG_{y,0}}(\bar{e}_{\alpha})$ is the abstract group $A(\bar{e})\times\Spec\bar{k}$. Recall that we assume (by base change) $\Frob$ acts trivially on the latter group. Recall that $\tau_{\bar{e}_{\alpha}}=\alpha$. Choose a cocycle that represent $\tau_{\bar{e}_{\alpha}}$, i.e. choose a representative of $\alpha$ in its conjugacy class. Then $\Frob$ acts on the group of geometric components of $Z_{G_{y,0}}(\bar{e}_{\alpha})$ by this representative of $\alpha$. Consequently the number of $\Frob$-stabilized geometric components of $Z_{G_{y,0}}(\bar{e}_{\alpha})$ is equal to the order of the centralizer of a representative of $\alpha$, which we denote as $|Z_{A(\bar{e})}(\alpha)|$. Namely
\begin{equation}\label{e2}
|Z_{G_{y,0}}(\bar{e}_{\alpha})|=|Z_{A(\bar{e})}(\alpha)|\cdot q^{\dim Z_{\bG_{y,0}}(\bar{e})}\cdot(1+O(q^{-1/2})).
\end{equation}
For the last term, note that the top-degree cohomology space of the Springer fiber above $\bar{e}$ is \cite[Theorem 1.5.1]{Yu16}
\[
H^{2\dim\CB_{\bar{e}}}(\CB_{\bar{e}})=\sum_{\eta'\in\Irr(A(\bar{e}))}\left(\eta'\right)^{\oplus\dim\rho_{\eta'}}.
\]
Since we assume that $\Frob$ acts trivially on the set of components of Springer fiber above $\bar{e}$ and thus trivially on $H^{2\dim\CB_{\bar{e}}}(\CB_{\bar{e}})$, it acts on $H^{2\dim\CB_{\bar{e}}}(\CB_{\bar{e}_{\alpha}})$ by the action of $\alpha$ on the representations $\eta'$.
By Grothendieck-Lefschetz trace formula (\ref{GL}) this implies
\begin{equation}\label{e3}
|\CB_{\bar{e}_{\alpha}}(k)|=q^{\dim\CB_{\bar{e}}}\cdot\left(\sum_{\eta'\in\Irr(A(\bar{e}))}\dim\rho_{\eta'}\cdot\Tr(\eta'(\alpha))\right)\cdot(1+O(q^{-1/2})).
\end{equation}
Plugging (\ref{e1}), (\ref{e2}) and (\ref{e3}) into (\ref{e0}) and then plug (\ref{e0}) into (\ref{regterm}), the contribution for each of the $|W|/|W_y|$ double cosets is, up to a $(1+O(q^{-1/2}))$-factor,
\[q^D\cdot\sum_{\eta'\in\Irr(A(\bar{e}))}\sum_{\alpha\in A(\bar{e})/conj}\frac{1}{|Z_{A(\bar{e})}(\alpha)|}\cdot\Tr(\eta(\alpha))\cdot\Tr(\eta'(\alpha))\cdot\dim\rho_{\eta'}\]
\[=q^D\cdot\sum_{\eta'\in\Irr(A(\bar{e}))}\delta_{\eta\eta'}\cdot\dim\rho_{\eta'}=q^D\cdot\dim\rho_{\eta}.\]
where
\[
\def\arraystretch{1.4}
\begin{array}{lrl}
D&:=&-\frac{1}{2}\dim\Ad(\bG_{y,0})\bar{e}+\dim{}^{\lambda}_{\ge0}\bG_{y,0}-\dim Z_{\bG_{y,0}}(\bar{e})+\dim\CB_{\bar{e}}\\
&=&-\frac{1}{2}\dim\Ad(\bG_{y,0})\bar{e}+\dim{}^{\lambda}_{\ge0}\bG_{y,0}-\dim Z_{\bG_{y,0}}(\bar{e})+\frac{1}{2}\dim Z_{\bG_{y,0}}(\bar{e})-\frac{1}{2}\rank\bG_{y,0}\\
&=&-\frac{1}{2}\dim\bG_{y,0}+\dim{}^{\lambda}_{\ge0}\bG_{y,0}-\frac{1}{2}\rank\bG_{y,0}=0.
\end{array}
\]
Hence $q^D=1$. Since there are $|W|/|W_y|$ many double cosets in $G_{y,\ge 0}\bsl G/P$. This proves (\ref{A7}) and finishes the proof of Theorem \ref{Steinberg}.\end{proof}\hp

\begin{corollary}\label{St2} We have $\dim\CX_{y,\gamma}=\dim\CX_{\gamma}$, and the number of $\CT$-orbits of top-dimensional components of $\CX_{y,\gamma}$ is equal to $|W|/|W_y|$.
\end{corollary}\p

\begin{proof} Since $\depth(\gamma)>0$, by Lemma \ref{TN}(iii) the affine Springer fiber $\CX_{y,\gamma}$ can be stratified into $\CX_{y,\bar{e},\gamma}$ where $\bar{e}$ runs over nilpotent orbits in $\Lie\bG_{y,0}$. By Theorem \ref{Steinberg}, we have $\dim\CX_{y,\bar{e},\gamma}\le\dim\CX_{\gamma}$ with equality iff $\bar{e}$ is regular nilpotent in $\Lg_{y,0}$. When $\bar{e}$ is the regular nilpotent, it is contained in a unique Borel subalgebra and consequently $\rho_{triv}$ is the trivial representation of $W_y$ and has dimension $1$. The first two statements then follows from Theorem \ref{Steinberg}.
\end{proof}

We remark that when $y=o$ is the hyperspecial vertex the corollary is well-known. However it's only until \cite[Cor 4.16.2]{Ngo10} that we know $\CX_{o,\gamma}$ is equi-dimensional (i.e. irreducible). We have no idea whether the same could be true for general $y$.\p

Now we explain how Theorem \ref{Steinberg} reveals a similarity between components of $\CX_{\gamma}$ and components of the Steinberg variety. We have seen in Section \ref{secgeom} that the fiber of $\CX_{\gamma}\ra\CX_{y,\gamma}$ above a (geometric) point on $\CX_{y,\gamma}$ is a Springer fiber. In fact, let \[X:=\{(\bar{g},B)\;|\;B\subset\bG_{y,0}\text{ a Borel. }\bar{e}\in\Lie B\}\ra\Lie\bG_{y,0}\]
be the Grothendieck-Springer resolution. This induces $[X/\bG_{y,0}]\ra[\Lie\bG_{y,0}/\bG_{y,0}]$. We have a natural identification 
\begin{equation}\label{strata1}
\CX_{\gamma}=\CX_{y,\gamma}\times_{[\Lie\bG_{y,0}/\bG_{y,0}]}[X/\bG_{y,0}].
\end{equation}
On the other hand, consider
\begin{equation}\label{strata2}
X_{\bar{e}}:=X\times_{\Lie\bG_{y,0}}\Ad(\bG_{y,0})\bar{e}=\{(\bar{g},B)\;|\;B\subset\bG_{y,0}\text{ a Borel. }\bar{e}\in\Ad(\bG_{y,0})\bar{e}\cap\Lie B\}.
\end{equation}
We have
\begin{equation}\label{strata3}
\CX_{y,\bar{e},\gamma}=\CX_{y,\gamma}\times_{[\Lie\bG_{y,0}/\bG_{y,0}]}[\Ad(\bG_{y,0})\bar{e}/\bG_{y,0}].
\end{equation}
Write $\CX_{\bar{e},\gamma}:=\CX_{\gamma}\times_{\CX_{y,\gamma}}\CX_{y,\bar{e},\gamma}$. (We warn the reader that $\bar{e}\in\Lg_{y,0}$ but not $\Lg_{x,0}$ in general. In particular this is not $\CX_{x,\bar{e},\gamma}$.) It is characterized as a locally closed sub-ind-variety of $\CX_{\gamma}$ by
\[\CX_{\bar{e},\gamma}(k)=\{g\in G/G_{x,\ge 0}\;|\;\Ad(g^{-1})\gamma\in\Lg_{x,\ge 0}\cap\Lg_{y,\ge0}^{(\bar{e})}\}.\]
As $\CX_{y,\gamma}$ can be stratified into $\CX_{y,\bar{e},\gamma}$, $\CX_{\gamma}$ is also stratified into $\CX_{\bar{e},\gamma}$ where $\bar{e}$ runs over nilpotent orbits in $\Lie\bG_{y,0}$. This is a $\CT$-stable stratification. One may ask how many among the $|W|$ orbits of components sit on $\CX_{\bar{e},\gamma}$. The answer is\p

\begin{corollary} The number of $\CT$-orbits of components of $\CX_{\gamma}$ whose generic points sit in $\CX_{\bar{e},\gamma}$ is equal to
\[
\frac{|W|}{|W_y|}\sum_{\eta\in\Irr(A(\bar{e}))}\left(\dim\rho_{\eta}\right)^2.
\]
That is, $\frac{|W|}{|W_y|}$ times the sum of the squares of the dimensions of the Springer representations of $W_y$ that are associated to some local system on $\Ad(\bG_{y,0})\bar{e}$.
\end{corollary}

\begin{proof}
The number of such orbits is $\dim H^{2d+2d_{\bar{e}}}(\CX_{\bar{e},\gamma}\times\Spec\bar{k})^{\CT}$. Since $\CX_{y,\bar{e},\gamma}$ has dimension $d$ and the fibers of $\pi:\CX_{\bar{e},\gamma}\times\Spec\bar{k}\ra\CX_{y,\bar{e},\gamma}\times\Spec\bar{k}$ has dimension $d_{\bar{e}}$, we have
\begin{equation}\label{top1}
H^{2d+2d_{\bar{e}}}(\CX_{\bar{e},\gamma}\times\Spec\bar{k},\underline{\mathbb{Q}_{\ell}})=H^{2d}(\CX_{y,\bar{e},\gamma}\times\Spec\bar{k},R^{2d_{\bar{e}}}\pi_*(\underline{\mathbb{Q}_{\ell}})).
\end{equation}
The top-degree derived push-forward $R^{2d_{\bar{e}}}\pi_*$ gives the monodromy of the components of the fibers of $\pi$. However, combining (\ref{strata1}), (\ref{strata2}) and (\ref{strata3}) we see that $\pi$ is a pullback of the Springer fiber map $[X_{\bar{e}}/\bG_{y,0}]\ra [\Ad(\bG_{y,0})\bar{e}/\bG_{y,0}]$. Springer theory says \cite[Theorem 1.5.1]{Yu16}
\begin{equation}\label{top2}
R^{2d_{\bar{e}}}\pi_*(\underline{\Q_{\ell}})=\sum_{\eta\in\Irr(A(\bar{e}))}\CL_{\eta}^{\oplus\dim\rho_{\eta}}.
\end{equation}
Combining (\ref{top1}) and (\ref{top2}) and taking $\CT$-invariant we arrive at
\begin{equation}\label{top3}
H^{2d+2d_{\bar{e}}}(\CX_{\bar{e},\gamma}\times\Spec\bar{k},\underline{\mathbb{Q}_{\ell}})^{\CT}=\sum_{\eta\in\Irr(A(\bar{e}))}H^{2d}(\CX_{y,\bar{e},\gamma}\times\Spec\bar{k},\,\CL_{\eta}^{\oplus\dim\rho_{\eta}})^{\CT}.
\end{equation}
According to Theorem \ref{Steinberg}, this implies the corollary.
\end{proof}\hp

The fact that the numbers in the corollary add up to $|W|$ corresponds to the fact that each representations of the Weyl group $W_y$ appears exactly once as a Springer representation. When $y=o$ is hyperspecial, $\bG_o=\bG$, and the number is the same as the number of components of the Steinberg variety of $\bG$ above the orbit $\Ad(\bG)\bar{e}$. Because of this, we expect an explicit construction of the orbits of components in Theorem \ref{main}, and expect the construction to be related to that for the Steinberg variety. We close by making two conjectures motivated by such expectation.

\begin{conjecture} Let $\gamma=t\gamma_0$ be as before. Suppose $\bG$ is simply connected. Consider $H^*(\CX_{\gamma}\times\Spec\bar{k})^{\CT}$ as a representation of the affine Weyl group \cite{Lu96}. Then its semi-simplification is the pull-back of the regular representation $\bar{\mathbb{Q}}_{\ell}[W]$.
\end{conjecture}

\begin{conjecture} For any $\gamma\in\Lg$ regular semisimple, the number of $\CT$-orbits of components of $\CX_{\gamma}\times\Spec\bar{k}$ is less than or equal to the order of the Weyl group. The equality holds also when $\depth(\gamma)=1$.
\end{conjecture}\p

\appendix
\section{Assumption on $\mathrm{char}(k)$}\label{secchar}

Firstly, we require that $\mathrm{char}(k)>\rank\bG+1$. This will meet the assumption in \cite[Theorem 61]{Mc04} regarding orbital integrals over $F$ so that there are only finitely nilpotent orbits, and any orbital integral of a function in $C_c^{\infty}(\Lg)$ converges. The assumption in particular says $\mathrm{char}(k)$ is very good for $\bG$ and implies the existence of a $\bG$-invariant bilinear form on $\Lie\bG$. The same assumption above also implies that any isogeny of $\bG$ induces an isomorphism on the Lie algebra, and that Jordan decomposition holds for $\bG$ over $F$ \cite[Proposition 48]{Mc04}. Lastly, it implies that $p$ does not divide the order of the Weyl group, which in turns implies that all $F$-tori in $\bG\times\Spec F$ are tame, that is, any subtorus (equivalently, any maximal subtorus) of $\bG$ defined over $F$ is split over some tamely ramified extension of $F$. We also assume that the hypotheses in \cite[Section 4.2]{De02b} hold; those hypotheses hold when $\mathrm{char}(k)$ is larger than a constant multiple of $\rank\bG$.

\section{Reduction to finite fields}\label{appred}

In this appendix we prove

\begin{proposition} Let $N$ be an integer. If Theorem \ref{main} holds when the ground field $k$ is any finite field of characteristic $p>N$. Then Theorem \ref{main} also holds for any field $k$ of characteristic zero or $p>N$.
\end{proposition}

\begin{proof} By conjugation we may assume that the maximal torus $Z_{\bG}(\gamma)$ is defined over $k(t)\subset k((t))=F$. We may also replace $\gamma$ by a sufficiently ($t$-adically) close element $\gamma'\in(\Lie Z_{\bG}(\gamma))(k(t))$ in the Lie algebra of the same torus. That $\gamma$ and $\gamma'$ is sufficiently close implies that for a large subvariety $X\subset\CX$ of the affine flag variety, $\CX_{\gamma}\cap X=\CX_{\gamma'}\cap X$. By \cite[Proposition 2.1]{KL88}, $\CX_{\gamma}$ (resp. $\CX_{\gamma'}$) is a union of translates of $\CX_{\gamma}\cap X$ (resp. $\CX_{\gamma'}\cap X$) by the centralizer once $X$ is large enough. This implies $\CX_{\gamma}=\CX_{\gamma'}$. We now replace $\gamma$ by $\gamma'$ and suppose $\gamma\in(\Lie Z_{\bG}(\gamma))(k(t))$. There is then a subfield $k'\subset k$ that is finitely generated (as a field) over its prime field such that $Z_{\bG}(\gamma)$ is defined over $k'(t)$ and $\gamma\in(\Lie Z_{\bG}(\gamma))(k(t))$. We now replace $k$ by $k'$ and assume $k$ is finitely generated over its prime field.\p

Let $S\subset k$ be a subring that is finitely generated as a ring over $\Z$ or $\mathbb{F}_p$, and such that $k$ is the quotient field of $S$. Let $R:=\Spec S$, on which the generic point is $\Spec k$. Since $\bG$ is split, we can extend it to be a split reductive group scheme over $R$. In the rest of the proof we will ``shrink $R$'' several times, which means inverting finitely many non-zero elements in $S$; every step of construction below is understood to begin with an implicit ``by shrinking $R$ if necessary,'' so that anything that is true on a Zariski open is true.\p

We may assume that any closed point on $R$ has characteristic $>N$. Choose any extension of $\gamma\in(\Lie\bG)(k(t))$ to an element in $(\Lie\bG)(S(t))$. We suppose (by shrinking) that $Z_{\bG}(\gamma)\subset\bG$ is a maximal torus over $R$. We may then extend the affine Springer fiber $\CX_{\gamma}$ over $R$, so that the fiber on each closed point is indeed an affine Springer fiber. This is done as follows: The affine Springer fiber is a closed sub-ind-variety of the affine flag variety for $\bG$, which can be constructed (say following \cite[Theorem 1.4]{PR08}) as a locally closed subscheme in a product of affine Grassmannians of $GL_n$. The latter is constructed as an inductive limit of Springer fibers of $GL_{n'}$, which is a closed subscheme of Grassmannian of $GL_n$. The last object is defined over $\Z$ and thus over $R$. The construction cutting out the affine Springer fiber within the product of affine Grassmannians works on a sufficiently small open subset of $R$.\p

Let $\Lambda_{\gamma}:=\{\lambda(t)\;|\;\lambda:\mathbb{G}_m/_{F^{ur}}\ra Z_{\bG\times\Spec F^{ur}}(\gamma)\text{ a cocharacter}\}$. Then $\Lambda_{\gamma}$ is an \'{e}tale sub-group scheme of $\CT$ over $\Spec k$ that acts freely on the $\CX_{\gamma}\times\Spec k$ (see e.g. \cite[3.3.2]{Yu16}). We extend $\Lambda_{\gamma}$ to $R$, and form the quotient $\Lambda_{\gamma}\bsl \CX_{\gamma}$ over $R$. Over $\Spec k$, there is a larger \'{e}tale sub-group scheme $\til{\Lambda}_{\gamma}\subset\CT$ \cite[3.4.11]{Yu16} that contains $\Lambda_{\gamma}$ with finite index such that $\til{\Lambda}_{\gamma}\sra\pi_0(\CT)$ is surjective. Since $\CT$ acts on $H^*(\CX_{\gamma}\times\Spec\bar{k})$ through $\pi_0(\CT)$, we have $H^*(\CX_{\gamma}\times\Spec\bar{k})^{\CT}=H^*((\til{\Lambda}_{\gamma}\bsl\CX_{\gamma})\times\Spec\bar{k})$, where $\til{\Lambda}_{\gamma}\bsl\CX_{\gamma}$ is constructed as a finite quotient of $\Lambda_{\gamma}\bsl \CX_{\gamma}$. We now extend $\til{\Lambda}$ over $R$ and have $\til{\Lambda}_{\gamma}\bsl\CX_{\gamma}$ over $R$. Since $\til{\Lambda}_{\gamma}\sra\pi_0(\CT)$ holds over the generic point $\Spec k$ we may assume that it holds on all points on $R$. Since affine Springer fibers have an algebraic dimension formula \ref{Bezr}, we may assume all fibers (as affine Springer fibers) have the same dimension $d$.\p

Let $\pi:\til{\Lambda}_{\gamma}\bsl\CX_{\gamma}\ra R$ be the natural map. Then $R^{2d}\pi_*\underline{\Ql}$ is a constructible sheaves on $R$, which we may assume to be a local system. By proper base change, we would like to prove that this local system on the generic point $\Spec k$ is constant with rank $|W|$. Theorem \ref{main} shows that this local system is constant with rank $|W|$ when restricted to any closed point on $R$. The result now follows from a higher-dimensional Chebotarev density theorem.\end{proof}

We remark that the same proposition holds for Theorem \ref{Steinberg}. The same proof applies as the nilpotent element $\bar{e}$ and the component group $A(\bar{e})$ both extend to $R$.



\p
\bibliographystyle{amsalpha}
\bibliography{biblio}

\def\cfgrv#1{\ifmmode\setbox7\hbox{$\accent"5E#1$}\else
  \setbox7\hbox{\accent"5E#1}\penalty 10000\relax\fi\raise 1\ht7
  \hbox{\lower1.05ex\hbox to 1\wd7{\hss\accent"12\hss}}\penalty 10000
  \hskip-1\wd7\penalty 10000\box7}
\providecommand{\bysame}{\leavevmode\hbox to3em{\hrulefill}\thinspace}
\providecommand{\MR}{\relax\ifhmode\unskip\space\fi MR }
\providecommand{\MRhref}[2]{%
  \href{http://www.ams.org/mathscinet-getitem?mr=#1}{#2}
}
\providecommand{\href}[2]{#2}
\begin{thebibliography}{{Yun}16}

\bibitem[Bez96]{Be96}
Roman Bezrukavnikov, \emph{The dimension of the fixed point set on affine flag
  manifolds}, Math. Res. Lett. \textbf{3} (1996), no.~2, 185--189. \MR{1386839}

\bibitem[BT72]{BT72}
F.~Bruhat and J.~Tits, \emph{Groupes r\'eductifs sur un corps local}, Inst.
  Hautes \'Etudes Sci. Publ. Math. (1972), no.~41, 5--251. \MR{0327923}

\bibitem[BT84]{BT84}
\bysame, \emph{Groupes r\'eductifs sur un corps local. {II}. {S}ch\'emas en
  groupes. {E}xistence d'une donn\'ee radicielle valu\'ee}, Inst. Hautes
  \'Etudes Sci. Publ. Math. (1984), no.~60, 197--376. \MR{756316}

\bibitem[Car93]{Ca93}
Roger~W. Carter, \emph{Finite groups of {L}ie type}, Wiley Classics Library,
  John Wiley \& Sons, Ltd., Chichester, 1993, Conjugacy classes and complex
  characters, Reprint of the 1985 original, A Wiley-Interscience Publication.
  \MR{1266626}

\bibitem[CG10]{CG10}
Neil Chriss and Victor Ginzburg, \emph{Representation theory and complex
  geometry}, Modern Birkh\"auser Classics, Birkh\"auser Boston, Inc., Boston,
  MA, 2010, Reprint of the 1997 edition. \MR{2838836}

\bibitem[DeB02a]{De02a}
Stephen DeBacker, \emph{Homogeneity results for invariant distributions of a
  reductive {$p$}-adic group}, Ann. Sci. \'Ecole Norm. Sup. (4) \textbf{35}
  (2002), no.~3, 391--422. \MR{1914003 (2003i:22019)}

\bibitem[DeB02b]{De02b}
\bysame, \emph{Parametrizing nilpotent orbits via {B}ruhat-{T}its theory}, Ann.
  of Math. (2) \textbf{156} (2002), no.~1, 295--332. \MR{1935848 (2003i:20086)}

\bibitem[DK06]{DK06}
Stephen DeBacker and David Kazhdan, \emph{Stable distributions supported on the
  nilpotent cone for the group {$G_2$}}, The unity of mathematics, Progr.
  Math., vol. 244, Birkh\"auser Boston, Boston, MA, 2006, pp.~205--262.
  \MR{2181807 (2006h:22012)}

\bibitem[GKM04]{GKM04}
Mark Goresky, Robert Kottwitz, and Robert Macpherson, \emph{Homology of affine
  {S}pringer fibers in the unramified case}, Duke Math. J. \textbf{121} (2004),
  no.~3, 509--561. \MR{2040285}

\bibitem[KL88]{KL88}
D.~Kazhdan and G.~Lusztig, \emph{Fixed point varieties on affine flag
  manifolds}, Israel J. Math. \textbf{62} (1988), no.~2, 129--168. \MR{947819
  (89m:14025)}

\bibitem[Lus96]{Lu96}
George Lusztig, \emph{Affine {W}eyl groups and conjugacy classes in {W}eyl
  groups}, Transform. Groups \textbf{1} (1996), no.~1-2, 83--97. \MR{1390751}

\bibitem[McN04]{Mc04}
George~J. McNinch, \emph{Nilpotent orbits over ground fields of good
  characteristic}, Math. Ann. \textbf{329} (2004), no.~1, 49--85. \MR{2052869}

\bibitem[MP94]{MP94}
Allen Moy and Gopal Prasad, \emph{Unrefined minimal {$K$}-types for {$p$}-adic
  groups}, Invent. Math. \textbf{116} (1994), no.~1-3, 393--408. \MR{1253198
  (95f:22023)}

\bibitem[Ng{\^o}10]{Ngo10}
Bao~Ch{\^a}u Ng{\^o}, \emph{Le lemme fondamental pour les alg\`ebres de {L}ie},
  Publ. Math. Inst. Hautes \'Etudes Sci. (2010), no.~111, 1--169. \MR{2653248
  (2011h:22011)}

\bibitem[PR08]{PR08}
G.~Pappas and M.~Rapoport, \emph{Twisted loop groups and their affine flag
  varieties}, Adv. Math. \textbf{219} (2008), no.~1, 118--198, With an appendix
  by T. Haines and Rapoport. \MR{2435422}

\bibitem[RR72]{Ra72}
R.~Ranga~Rao, \emph{Orbital integrals in reductive groups}, Ann. of Math. (2)
  \textbf{96} (1972), 505--510. \MR{0320232 (47 \#8771)}

\bibitem[SS70]{SS70}
T.~A. Springer and R.~Steinberg, \emph{Conjugacy classes}, Seminar on
  {A}lgebraic {G}roups and {R}elated {F}inite {G}roups ({T}he {I}nstitute for
  {A}dvanced {S}tudy, {P}rinceton, {N}.{J}., 1968/69), Lecture Notes in
  Mathematics, Vol. 131, Springer, Berlin, 1970, pp.~167--266. \MR{0268192 (42
  \#3091)}

\bibitem[Tit79]{Tits}
J.~Tits, \emph{Reductive groups over local fields}, Automorphic forms,
  representations and {$L$}-functions ({P}roc. {S}ympos. {P}ure {M}ath.,
  {O}regon {S}tate {U}niv., {C}orvallis, {O}re., 1977), {P}art 1, Proc. Sympos.
  Pure Math., XXXIII, Amer. Math. Soc., Providence, R.I., 1979, pp.~29--69.
  \MR{546588 (80h:20064)}

\bibitem[{Yun}16]{Yu16}
Z.~{Yun}, \emph{{Lectures on Springer theories and orbital integrals}}, ArXiv
  e-prints (2016).

\end{thebibliography}
\p

\end{document}